\documentclass{article}

\usepackage{amsthm, amsmath, amssymb, mathtools}
\usepackage[T1]{fontenc}
\usepackage[utf8]{inputenc}
\usepackage{lmodern}
\usepackage[margin=3cm]{geometry}
\usepackage[shortlabels]{enumitem}
\usepackage{comment}
\usepackage{graphicx}
\usepackage{float,color}
\usepackage[normalem]{ulem}
\usepackage{epsf,epsfig,subfigure}
\usepackage[colorlinks=true]{hyperref}
\hypersetup{urlcolor=red,linkcolor=red ,citecolor=blue,colorlinks=true}

\usepackage{bbm, mathrsfs}

\usepackage{tikz}
\usetikzlibrary{intersections}

\newenvironment{system}[1]
{
	\left\{\begin{array}{#1}} 
	{ \end{array}\right.}

\makeatletter{}
\makeatother{}

\newcommand\dd{\mathrm{d}}
\newcommand\diff[1]{\frac{\dd\phantom{#1}}{\dd #1}}
\newcommand\R{\mathbb{R}}

\newcounter{stepnum}
\newcommand\step{\typeout{Use of step outside a stepping environment ! }}
\newenvironment{stepping}{
	\setcounter{stepnum}{1}
	\renewcommand\step{\medskip\noindent{\bf Step \thestepnum:~}\stepcounter{stepnum}}
}{\renewcommand\step{\typeout{Use of step outside a stepping environment !}}}

\newtheorem{thm}{Theorem}[section]
\newtheorem{lem}[thm]{Lemma}
\newtheorem{prop}[thm]{Proposition}

\theoremstyle{definition}

\newtheorem{defn}{\textbf{Definition}}[section]

\theoremstyle{remark}
\newtheorem{rem}{Remark}[section]
\newtheorem{assumption}{\textbf{Assumption}}

\allowdisplaybreaks

\numberwithin{equation}{section}

\begin{document}

\title{Sharp discontinuous traveling waves in a hyperbolic Keller--Segel equation}
\author{\textsc{Xiaoming Fu\thanks{The research of this author is supported by China Scholarship Council.} , Quentin Griette\footnote{Corresponding author.} and Pierre Magal}
	\\
	{\small \textit{Univ. Bordeaux, IMB, UMR 5251, F-33400 Talence, France}} \\
	{\small \textit{CNRS, IMB, UMR 5251, F-33400 Talence, France.}}\\
	{\small \textit{}} }
\maketitle

	\begin{abstract}
		In this work we describe a hyperbolic model with cell-cell repulsion with a dynamics in the population of cells. More precisely, we consider a population of cells producing a field (which we call ``pressure'') which induces a motion of the cells following the opposite of the gradient. The field indicates the local density of population and we assume that cells try to avoid crowded areas and prefer locally empty spaces which are far away from the carrying capacity. We analyze the well-posedness property of the associated Cauchy problem on the real line. We start from bounded initial conditions and we consider some invariant properties of the initial conditions such as the continuity, smoothness and monotony. We also describe in detail the behavior of the level sets near the propagating boundary of the solution and we find that an asymptotic jump is formed on the solution for a natural class of initial conditions. Finally, we prove the existence of sharp traveling waves for this model, which are particular solutions traveling at a constant speed, and argue that sharp traveling waves are necessarily discontinuous. This analysis is confirmed by numerical simulations of the PDE problem.
	\end{abstract}
	\section{Introduction}

	In this article we are concerned with the following diffusion equation with logistic source: 
	\begin{equation}\label{eq:main}
		\begin{cases}
			\partial _{t}u(t,x) -\chi\partial_x\bigl(u(t,x)\partial_x p(t,x)\bigr)=u(t,x)(1-u(t,x)), & t>0,\;x\in \mathbb{R},\\
			u(t=0, x)=u_0(x),
		\end{cases}\;\;
	\end{equation}
	where  $\chi>0$ is a {\it sensing coefficient} and $p(t, x) $ is an external pressure. Model \eqref{eq:main} describes the behavior of a population of cells $u(t,x)$ living in a one-dimensional habitat $x\in\mathbb R$, which undergo a logistic birth and death population dynamics, and in which individual cells follow the gradient of a field $p$. The constant $\chi$ characterizes the response of the cells to the effective gradient $p_x$. In this work we will consider the case where $p$ is itself determined by the state of the population $u(t,x)$ as
	\begin{equation}\label{eq:main-p}
		-\sigma^2\partial_{xx}p(t, x)+p(t,x)=u(t,x), \quad t>0, x\in\mathbb R.
	\end{equation}
	This corresponds to a scenario in which the field $p(t, x)$ is produced by the cells, diffuses to the whole space with diffusivity $\sigma^2$ (for $\sigma>0$), and vanishes at rate one. As a result cells are pushed away from crowded area to emptier region. 
	

A similar model has been successfully used in our recent work \cite{Fu2019} to describe the motion of cancer cells in a Petri dish in the context of cell co-culture experiments of Pasquier et al. \cite{Pas-Mag-Bou-Web-11}.  Pasquier et al. \cite{Pas-Mag-Bou-Web-11} cultivated two types of breast cancer cells to study the transfer of proteins between them in a study of multi-drug resistance.  It was observed that the two types of cancer cells  form segregated clusters of cells of each kind after a 7-day co-culture experiment (Figure \ref{fig:segregation} {\bf(a)}). In \cite{Fu2019}, the authors studied the segregation property of a model similar to \eqref{eq:main}--\eqref{eq:main-p}, set in a circular domain in two spatial dimensions $x\in\mathbb R^2$ representing a Petri dish. Starting from islet-like initial conditions representing cell clusters, it was numerically observed that the distribution of cells converges to a segregated state in the long run. 

One may observe that in such an experiment the cells are well fed. So there is no limitation for food. As explained in \cite{ducrot2011vitro}, the limitations are due to space and the contact inhibition of growth is involved. Therefore the right hand side of the \ref{eq:main} which is a logistic term (for simplicity) could possibly have the following 
$$
f(x)=\dfrac{\beta x}{1+\alpha x}-\mu x
$$ 
where $\beta$ is the division rate and $\mu$ is the mortality rate. We believe that our result hold for such a non-linearity and this is left for future work.

	\begin{figure}[H]
		\centering
		\begin{minipage}{0.48\textwidth}
			\noindent\centering
			{\bf (a)}\medskip 

			\noindent\includegraphics[width=\textwidth]{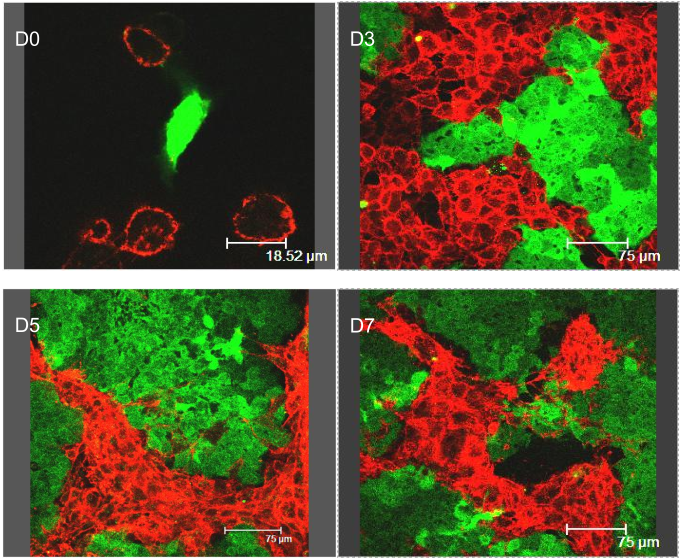}
		\end{minipage}\hfill
		\begin{minipage}{0.48\textwidth}
			\noindent\centering {\bf (b)}\medskip 
			
			\noindent\includegraphics[width=\textwidth]{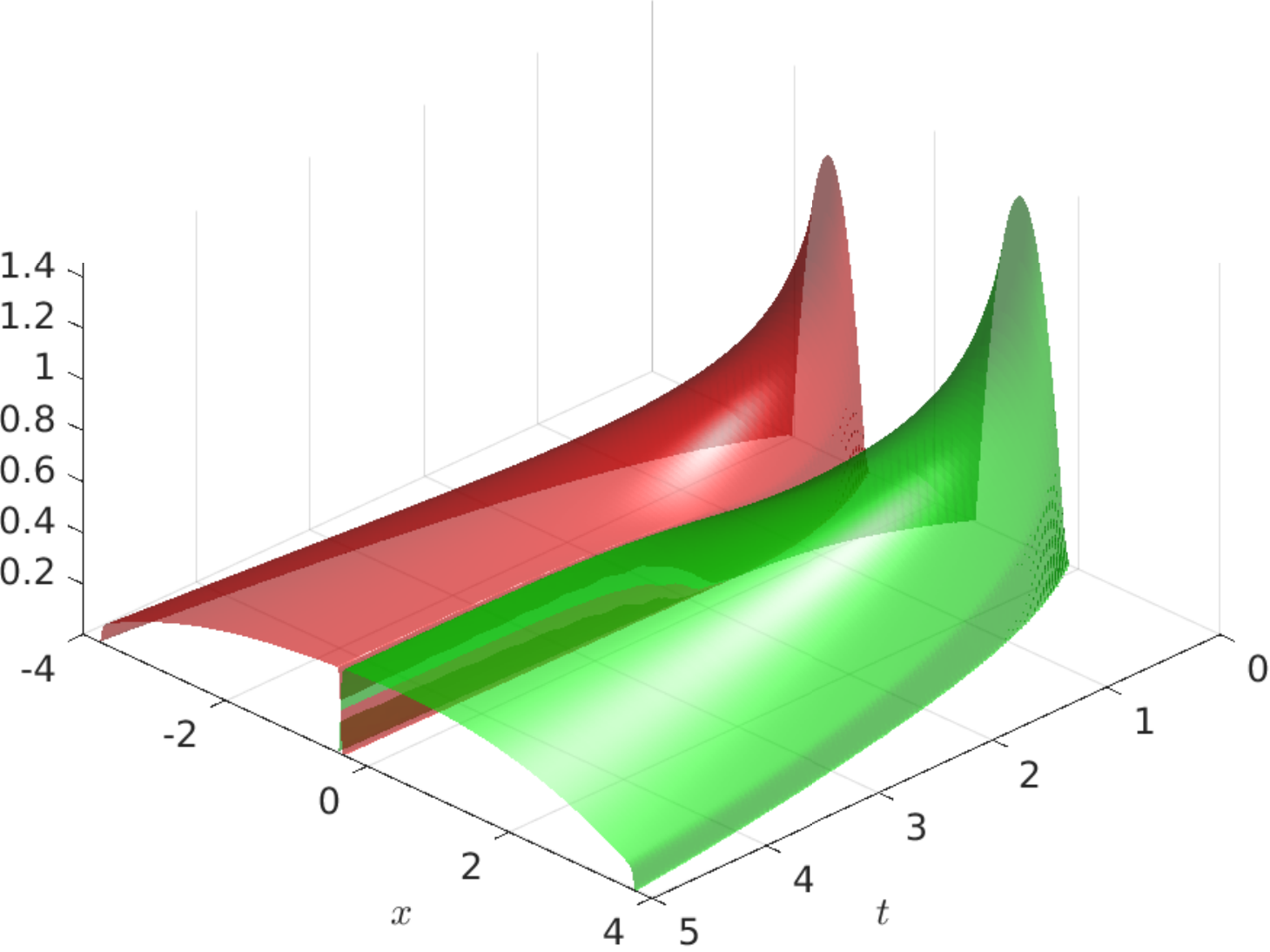}
		\end{minipage}
		\caption{\textit{\cite[Figure 1 and Figure 5 (b)]{Fu2019}. {\bf(a)} Direct immunodetection of P-gp transfers in co-cultures of sensitive (MCF-7) and resistant (MCF-7/Doxo) variants of the human breast cancer cell line. {\bf(b)} The temporal-spatial evolution of the two species in the 1D model. One can check that a discontinuity is forming near the front face of the green surface.}}\label{fig:segregation}
	\end{figure}

Strikingly, even before the two species come in contact, a sharp transition is formed between the space occupied by one species and the empty space being invaded (Figure \ref{fig:segregation} {\bf(b)}) and the distribution of cells looks like a very sharp traveling front. In an attempt to better understand the spatial behavior of cell populations growing in a Petri dish, in the present paper we investigate the mathematical properties of a simplified model for a single species on the real line. We are particularly interested in showing the existence of a sharp traveling front moving at a constant speed.

Recall that a traveling wave a special solution having the specific form 
$$
u(t, x)=U(x-ct), \text{ for a.e. } (t, x)\in\mathbb R^2,
$$
where the profile $U$ has the following behavior at $\pm\infty$:
$$
\lim_{z\to - \infty}U(z) =1,\quad \lim_{z\to \infty} U(z)=0. 
$$
The goal of this article is to investigate the sharp traveling namely 
$$
U(x)=0, \forall x > 0. 
$$
Moreover as it is represented in Figure \ref{Fig2}-(a) we will obtain the existence of such a wave with a discontinuity at $x=0$ for the profile $U$.  
\begin{figure}[H]
	\begin{center}
		\begin{tikzpicture}[line width=1pt]
		
		\draw (-7, 4.5) ..controls +(3, 0) and +(-1, 1) .. (-2, 3) -- (-2, 0);
		\path (-6.5, 2.5) node (line1) [anchor=north west] {discontinuous};
		\path (line1.south west) node [anchor=north west]{traveling wave};
		\draw[->] (-7, 0) -- (-1, 0) node [pos=0.5, below=5pt] {(a)};
		\draw[->] (-7, 0) -- (-7, 5);
		
		\draw (1, 4.5) ..controls +(3, 0) and +(-0.5, 2) .. (5, 2) ..controls +(0.5, -2) and +(-1, 0).. (6.5, 0);
		\path (1.5, 2.5) node (line1) [anchor=north west] {smooth};
		\path (line1.south west) node [anchor=north west] {traveling wave};
		\draw[->] (1, 0) -- (7, 0) node[pos=0.5, below=5pt] {(b)};
		\draw[->] (1, 0) -- (1., 5);
		\end{tikzpicture}
	\end{center}
	\vspace{-0.5cm}
	\caption{\textit{An illustration of two types of traveling wave solutions.}}
	\label{Fig2}
\end{figure}
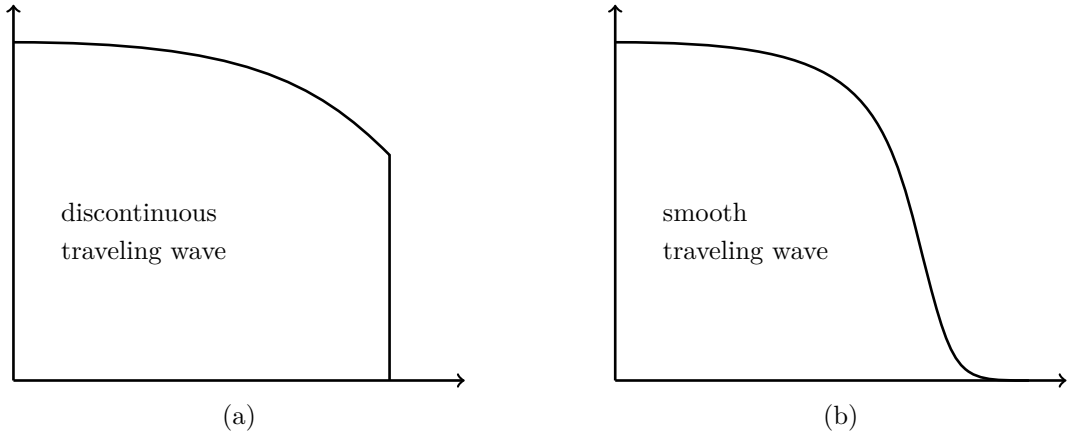

It can be noticed that, in the limit of slow diffusivity $\sigma\to 0$ (and under the simplifying assumption that $\chi=1$), we get $u(t,x)\equiv p(t,x)$ and \eqref{eq:main} is equivalent to an equation with  {\it porous medium-type diffusion} and logistic reaction
\begin{equation}\label{eq:porous_medium}
u_t - \frac{1}{2}(u^2)_{xx}=u(1-u).
\end{equation}
We refer to the monograph of V\'azquez \cite{Vazquez2007} for more result about porous medium equation. 

The propagation dynamics for this kind of equation was first studied, to the extent of our knowledge, by Aronson \cite{Aronson1980}, Atkinson, Reuter and Ridler-Rowe \cite{Atkinson1981}, and later by de Pablo and V\'azquez \cite{dePablo1991}, in the more general context of nonlinear diffusion 
\begin{equation}\label{eq:porous_medium_1}
u_t = (u^m)_{xx} + u(1-u), \text{ with } m>1.
\end{equation}
In Section \ref{Section3.3}, we observe that the discontinuous sharp traveling obtained in the present article converge (numerically) to the continuous profile described by Pablo and V\'azquez \cite{dePablo1991}.

	Our model can be included in the family of non-local advection models for cell-cell adhesion and repulsion. As pointed out by many biologists, cell-cell interactions do not only exist in a local scope, but a long-range interaction should be taken into account to guide the mathematical modeling.  Armstrong, Painter and Sherratt \cite{Armstrong2006} in their early work proposed a model (APS model) in which a local diffusion is added to the non-local attraction driven by the adhesion forces to describe the phenomenon of cell mixing, full/partial engulfment and complete sorting in the cell sorting problem. Based on the APS model, Murakawa and Togashi \cite{Murakawa2015} thought that the population pressure should come from the cell volume size instead of the linear diffusion. Therefore, the linear diffusion was changed into a nonlinear diffusion in order to capture the sharp fronts and the segregation in cell co-culture. Carrillo et al. \cite{Carrillo2019} recently proposed a new assumption on the adhesion velocity field and their model showed a good agreement in the experiments in the work of Katsunuma et al.  \cite{Katsunuma2016}. The idea of the long-range attraction and short-range repulsion can also be seen in the work of Leverentz, Topaz and Bernoff \cite{Leverentz2009}. They considered a non-local advection model to study the asymptotic behavior of the solution. By choosing a Morse-type kernel which follows the attractive-repulsive interactions, they found that the solution can asymptotically spread, contract (blow-up), or reach a steady-state.  Burger, Fetecau and Huang \cite{Burger2014} considered a similar non-local adhesion model with nonlinear diffusion, for which they investigated  the well-posedness and proved the existence of a compactly supported, non-constant steady state. Dyson et al. \cite{Dyson2010}  established the local existence of a classical solution for a non-local cell-cell adhesion model in spaces of uniformly continuous functions. For Turing and Turing-Hopf bifurcation due to the non-local effect, we refer to Ducrot et al. \cite{Duc-Fu-Mag-18} and Song et al. \cite{Song2019}. We also refer to Mogliner et al. \cite{Mogilner2003}, Eftimie et al. \cite{Eftimie2007}, Ducrot and Magal \cite{Duc-Mag-14}, Ducrot and Manceau \cite{ducrot2020one} for more topics on non-local advection equations. For the derivation of such models, we refer to the work of Bellomo et al. \cite{Bel-Bel-Nie-Sol-12} and Morale, Capasso and Oelschl\"ager \cite{Mor-Cap-Oel-05}.

%

	Discontinuous traveling waves in hyperbolic partial differential equations have appeared in the literature of  the recent few years. Travelling wave solutions with a shock or jump discontinuity have been found {\it e.g.} in models of malignant tumor cells (Marchant, Norbury and Perumpanani \cite{Marchant2000}, Harley {\it et al.} \cite{Harley2014} where the existence of discontinuous waves is proved by means of  geometric singular perturbation theory for ODEs) or chemotaxis (Landman, Pettet and Newgreen \cite{Landman2003} where both smooth and discontinuous traveling waves are found using phase plane analysis).
	Bouin, Calvez and Nadin \cite{Bouin-Calvez-Nadin14} considered the following  hyperbolic model
	$$
	\varepsilon^2\partial_{tt}\rho_\varepsilon+(1-\varepsilon^2F'(\rho_\varepsilon))\partial_t\rho_\varepsilon-\partial_{xx}\rho(t,x)=F(\rho_\varepsilon), 
	$$
	where the reaction term $ F $ is monostable. They identified two different regimes for the propagating behavior of solutions. In the first regime $\varepsilon^2F’(0)<1$, there exists a smooth traveling front (as in the Fisher-KPP case), whereas in the second regime $\varepsilon^2F’(0) > 1$ the traveling wave is discontinuous. In the critical case when $\varepsilon^2F'(0) = 1$, there exists a continuous traveling front with minimal speed $\sqrt{F'(0)}$ which may present a jump in the derivative.

	 The particular relation between the pressure $p(t,x)$ and the density $u(t,x)$ in \eqref{eq:main-p} strongly  reminds the celebrated model of chemotaxis studied by Patlak (1953) and Keller and Segel (1970) \cite{Pat-53, Kel-Seg-70, Kel-Seg-71} (parabolic-parabolic Keller-Segel model) and, more specifically, the parabolic-elliptic Keller-Segel model which is derived from the former by a quasi-stationary assumption on the diffusion of the chemical  \cite{Jae-Luc-92}. 
	{Indeed Equation \eqref{eq:main-p} can be formally obtained as the quasistatic  approximation of the following parabolic equation 
	\begin{equation*}
		\varepsilon \partial_t p(t, x)=\chi p_{xx}(t,x)+u(t,x)-p(t,x),
	\end{equation*}
	when $\varepsilon\to 0$.
	A rigorous derivation of the limit has been achieved in the  case of the Keller-Segel model by Carrapatoso and Mischler \cite{Car-Mis-17}. We also refer to Mottoni and Rothe \cite{mottoni1980singular} for such a result in the context of linear parabolic equations. We refer to \cite{Cal-Cor-08,Hil-Pai-09,Per-Dal-09} and the references therein for a mathematical introduction and biological applications.  In these models, the field $p(t,x)$ is interpreted as the concentration of a chemical produced by the cells rather than a physical pressure.}
	 One of the difficulties in attractive chemotaxis models is that two opposite forces compete to drive the behavior of the equations: the {\it diffusion} due to the random motion of cells, on the one hand, and on the other hand the {\it non-local advection} due to the {attractive} chemotaxis; the former tends to regularize and homogenize the solution, while the latter promotes cell aggregation and may lead to the blow-up of the solution in finite time \cite{Chi-84,Jae-Luc-92}. At this point let us mention that our study concerns {\it repulsive} cell-cell interaction with no diffusion, therefore no such blow-up phenomenon is expected in our study; however the absence of diffusion adds to the mathematical complexity of the study, because standard methods of reaction-diffusion equations cannot be employed here. Traveling waves for the (attractive) parabolic-elliptic Keller-Segel model were studied by Nadin, Perthame and Ryzhik \cite{Nadin2008}, who constructed these traveling wave by a bounded interval approximation of the 1D  system
	\begin{equation}\label{eq:K-S-attractant}
		\begin{system}{l}
			u_t + \chi\, (u p_x)_{x}= u_{xx} +u(1-u), \\
			-d\,p_{xx} +p = u,
		\end{system}
	\end{equation}
	set on the real line $x\in\mathbb R$, when the strength of the advection is not too strong $0< \chi<\min(1, d)$, and gave estimates on the speed of such a traveling wave: $ 2\leq c_* \leq 2 +\chi \sqrt{d}/(d-\chi) $.
	 
	 Since the pressure $p(t,x)$ is a non-local function of the density $u(t,x)$ in \eqref{eq:main-p}, the spatial derivative appears as a {\it non-local advection} term  in \eqref{eq:main}. In fact, our problem \eqref{eq:main}--\eqref{eq:main-p} can be rewritten as a transport equation in which the speed of particles is non-local in the density, 
	\begin{equation}\label{eq:main1}
		\begin{system}{l}
			\partial_t u(t,x) - \chi\partial_x(u(t,x)\partial_x(\rho\star u)(t,x))=u(t,x)(1-u(t,x)) \\
			u(t=0, x)=u_0(x),
		\end{system}
	\end{equation}
	where 
	\begin{equation}\label{eq:convol}
		\left( \rho\star u \right)(x)=\int_\mathbb{R}\rho(x-y)u(t,y)dy,\quad \rho(x) =\frac{1}{2\sigma}e^{-\frac{|x|}{\sigma}}.
	\end{equation} 
	 Traveling waves for a similar diffusive equation with logistic reaction have been investigated for quite general non-local kernels by Hamel and Henderson \cite{Hamel2020}, who considered the model 
	\begin{equation}\label{eq:adv-nonloc}
		u_t + (u\,(K\star u))_x=u_{xx}+u(1-u), 
	\end{equation}
	where $ K \in L^p(\R)$ is odd and  $p\in [1,\infty] $. Notice that the attractive parabolic-elliptic  Keller-Segel model \eqref{eq:K-S-attractant} is included in this framework by the particular choice  
	\begin{equation*}
		K(x) =- \chi\,{\rm sign}(x) e^{-|x|/\sqrt{d}}/\big(2\sqrt{d}\big). 
	\end{equation*}
	They proved a spreading result for this equation (initially compactly supported solutions to the Cauchy problem propagate to the whole space with constant speed) and explicit bounds on the speed of propagation. Diffusive non-local advection also appears in the context of swarm formation \cite{Mogilner1999}. Pattern formation for a model similar to \eqref{eq:adv-nonloc} by Ducrot, Fu and Magal \cite{Duc-Fu-Mag-18}. Let us mention that the inviscid equation \eqref{eq:main1} has been studied in a periodic cell by Ducrot and Magal \cite{Duc-Mag-14}. Other methods  have been established for conservative systems of interacting particles and their kinetic limit (Balagué et al. \cite{Bal-13}, Carrillo et al. \cite{Car-DiF-Fig-Lau-Sle-12}) based on gradient flows set on measure spaces; those are difficult to adapt here because of the logistic term. Finally we refer to \cite{Coville2005, Griette2019, Wang2007,Faye2015,Ducrot2014} for other examples of traveling waves in non-local reaction-diffusion equations.

	In this paper we focus on the particular case of \eqref{eq:main}--\eqref{eq:main-p} with $\sigma>0$ and $\chi>0$. The paper is organized as follows. In Section 2, we present our main results. In Section 3 we present  numerical simulations to illustrate our theoretical results.   In Section 4, we prove the propagation properties of the solution and describe the local behavior near the propagating boundary (see Proposition \ref{prop:separatrix} for definition), including the formation of a discontinuity for time-dependent solutions. In Section 5 we prove the existence of sharp traveling waves. We also prove that smooth traveling waves are necessarily positive, which shows that sharp traveling waves are necessarily singular (in this case, discontinuous). In particular, a solution starting from a compactly supported initial condition with polynomial behavior at the boundary can never catch such a smooth traveling wave.

\section{Main results and comments}

We begin by defining our notion of solution to equation \eqref{eq:main}.

\begin{defn}[Integrated solutions]\label{def:intsol}
	Let $u_0\in L^\infty(\mathbb R)$. A measurable function $u(t,x)\in L^\infty([0, T]\times \mathbb R)$ is an {\it integrated solution} to \eqref{eq:main} if the characteristic equation 
	\begin{equation}\label{eq:characteristics}
		\begin{system}{l}
			\frac{\dd\phantom{t}}{\dd t}h(t,x)=-\chi(\rho_x\star u)(t, h(t, x)) \\
			h(t=0, x)=x.
		\end{system}
	\end{equation}
	has a classical solution $h(t, x)$ ({\it i.e.} for each  $x\in\mathbb R$ fixed, the function $t\mapsto h(t,x) $ is in $ C^1([0, T], \mathbb R)$ and satisfies \eqref{eq:characteristics}), and for a.e. $x\in \mathbb R$, the function $t\mapsto u(t, h(t,x)) $ is in $C^1([0, T], \mathbb R)$ and satisfies 
	\begin{equation}\label{eq:integrated}
		\begin{system}{l}
			\frac{\dd \phantom{t}}{\dd t}u(t, h(t,x))=u(t, h(t,x))\big(1+\hat\chi(\rho\star u)(t, h(t, x))-(1+\hat\chi)u(t, h(t,x))\big), \\
			u(t=0, x)=u_0(x),
		\end{system}
	\end{equation}
	where $\hat\chi:=\frac{\chi}{\sigma^2}$.
\end{defn}
We define weighted space $L^1_\eta(\mathbb R)$ as follows
\begin{equation*}
	L^1_{\eta}(\R):= \bigg\lbrace f:\R\to \R \text{ measurable}\,\bigg| \int_\R|f(x)| e^{-\eta|x|}\dd x <\infty \bigg\rbrace. 
\end{equation*}
$L^1_\eta(\mathbb R)$  is a Banach space endowed with the norm 
$$
\Vert f\Vert_{L^1_\eta}:=\frac{\eta}{2}\int_{\mathbb R}|f(y)|e^{-\eta|y|}\dd y.
$$
 We first recall some results concerning the existence of integrated solutions for equation \eqref{eq:main} in Theorem \ref{thm:well-posed}, Proposition \ref{prop:regularity} and Theorem \ref{thm:ltb}. We prove those results in the companion paper \cite{Fu2020}.  
\begin{thm}[Well-posedness ]\label{thm:well-posed} 
	Let $u_0\in L^\infty_+(\mathbb R)$ and fix $\eta>0$. There exists $\tau^*(u_0)\in (0, +\infty]$ such that for all $\tau\in (0, \tau^*(u_0))$, there exists a unique integrated solution $u\in C^0([0, \tau], L^1_\eta(\mathbb R))$ to \eqref{eq:main} which satisfies $u(t=0, x)=u_0(x)$. Moreover $u(t,\cdot)\in L^\infty(\mathbb R)$ for each $t\in[0, \tau^*(u_0))$ and the map $t\in[0, \tau^*(u_0))\mapsto T_t u_0:=u(t,\cdot)$ is a semigroup which is continuous for the $L^1_\eta(\mathbb R)$-topology. 
	The map $u_0\in L^\infty(\mathbb R)\mapsto T_tu_0\in L^1_\eta(\mathbb R) $ is continuous.

	Finally, if $0\leq u_0(x)\leq 1$, then $\tau^*(u_0)=+\infty$ and $0\leq u(t, \cdot)\leq 1 $ for all $t>0$.
\end{thm}
The next result concerns the preservation properties satisfied by the solutions of \eqref{eq:main} (see  \cite[Proposition 2.2]{Fu2020}). 
\begin{prop}[Regularity of solutions]\label{prop:regularity}
	Let $u(t, x)$ be an integrated solution to \eqref{eq:main}. 
	\begin{enumerate}
		\item if $u_0(x) $ is continuous, then $u(t, x)$ is continuous for each $t>0$. 
		\item if $u_0(x) $ is monotone, then $u(t, x)$ has the same monotony for each $t>0$.
		\item if $u_0(x)\in C^1(\mathbb R)$, then $u\in C^1([0, T]\times \mathbb R) $ and $u$ is then a classical solution to \eqref{eq:main}.
	\end{enumerate}
\end{prop}
In this following theorem we consider the long-time behavior of some solutions to \eqref{eq:main} (see  \cite[Theorem 2.3]{Fu2020}).  
\begin{thm}[Long-time behavior]\label{thm:ltb}
	Let $0\leq u_0(x)\leq 1$ be a nontrivial non-negative initial condition and $u(t, x)$ be the corresponding integrated solution. Then $0\leq u(t, x)\leq 1$ for all $t>0$ and $x\in\mathbb R$. If moreover there exists $\delta>0$ such that $\delta\leq u_0(x)\leq 1$ then  
	\begin{equation*}
		u(t, x) \to 1,\text{ as } t\to\infty 
	\end{equation*}
	and the convergence holds  uniformly in $x\in \mathbb R$. 
\end{thm}

 We now arrive at the main interest of the paper, which is to describe the spatial dynamics of solutions to \eqref{eq:main}--\eqref{eq:main-p}. To get insight about the asymptotic propagation properties  of  the solutions, we focus on initial conditions  whose support is bounded towards $+\infty$. If the behavior of the initial condition in a neighbourhood of the boundary of the support is polynomial, we can establish a precise estimate of the  location of the level sets  relative to the position of the rightmost positive point.  Our first assumption requires that the initial condition is supported in $(-\infty, 0]$.
\begin{assumption}[Initial condition]\label{assum:initcond-basic}
	We assume that $u_0(x)$ is a continuous function satisfying 
	\begin{align*}
		0\leq u_0(x)&\leq 1, & &\forall x\in\mathbb R, \\
		u_0(x)&=0,	& &\forall x\geq 0, \\
		u_0(x)&>0, & &\forall x\in (-\delta_0, 0),
	\end{align*}
	for some $\delta_0>0$.
\end{assumption}
Under this assumption we show that $u$ is propagating to the right. 
\begin{prop}[The separatrix]\label{prop:separatrix}
	Let $u_0(x)$ satisfy Assumption \ref{assum:initcond-basic}, and $h^*(t):=h(t,0)$ be the separatrix. Then $h^*(t)$ stays at the rightmost boundary of the support of $u(t, \cdot)$, i.e.
	\begin{enumerate}[label={\rm (\roman*)}]
		\item \label{item:ultimatelytrivial}
			we have
			\begin{equation}
				\qquad u(t, x)=0\text{ for all }  x\geq h^*(t),
			\end{equation}
		\item \label{item:hboundary}
			for each $t>0$ there exists $\delta>0$ such that  
			\begin{equation}
				 \qquad u(t, x)>0 \text{ for all }x\in \big(h^*(t)-\delta, h^*(t)\big).
			\end{equation}
	\end{enumerate}
	Moreover, $u$ is propagating to the right i.e.
	\begin{equation*}
		\qquad\frac{\dd\phantom{t}}{\dd t} h^*(t)>0 \text{ for all }t>0.
	\end{equation*}
\end{prop}
We precise the behavior of the initial condition in a neighbourhood of $0$ and estimate the steepness of $u$ in positive time.
\begin{assumption}[Polynomial behavior near 0]\label{assum:initcond-steep}
	In addition to Assumption \ref{assum:initcond-basic}, we require that there exists $\alpha\geq 1$ and $\gamma>0$ such that
	\begin{equation*}
		u_0(x)\geq \gamma |x|^\alpha, \qquad \forall x\in (-\delta, 0).
	\end{equation*}
\end{assumption}

\begin{thm}[Formation of a discontinuity]\label{thm:discont} 
	Let $u_0(x)$ satisfy Assumptions \ref{assum:initcond-basic} and \ref{assum:initcond-steep} and $u(t,x)$ solve  \eqref{eq:main} with $u(t=0, x)=u_0(x)$. For all $\delta>0$ we have 
	\begin{equation}\label{eq:discont}
		\limsup_{t\to+\infty}\sup_{x\in (h^*(t)-\delta, h^*(t))}u(t,x)\geq \frac{1}{1+\hat\chi+\alpha\chi}>0.
	\end{equation}
	More precisely, define the level set
	\begin{equation*}
		\xi(t, \beta):=\sup\{x\in\mathbb R\,|\, u(t, x)= \beta\},
	\end{equation*} 
	for all $t>0$ and  $ 0<\beta<\frac{1}{1+\hat\chi+\alpha\chi}  $.
	Then, for each $0<\beta<\frac{1}{1+\hat\chi+\alpha\chi}$, the distance between $\xi(t, \beta)$ and the separatrix is decaying exponentially fast: 
	\begin{equation}\label{eq:levelset}
		h^*(t)-\left(\frac{\beta}{\gamma}\right)^{\frac{1}{\alpha}}e^{-\frac{\eta}{2\alpha}t} \leq \xi(t, \beta)\leq h^*(t),
	\end{equation}
	where $\eta\in (0,1) $ is given in Proposition \ref{prop:levelset} and $\hat\chi=\frac{\chi}{\sigma^2}$.
\end{thm}
\begin{figure}[H]
	\begin{center}
		\begin{tikzpicture}[line width=1pt]
			\draw[name path=front1] (0, 3.5) ..controls +(2, -0.5) and +(0, 1) .. (3,0)  node [anchor=north west]{$h^*(t_1)$};
			\draw[name path=front2] (5, 3.5) ..controls +(2, 0) and +(0, 3) .. (8,0) node[anchor=north west] {$h^*(t_2)$};
			\draw[name path=beta, dashed] (0, 1.5) -- (10, 1.5) node[anchor=south east]{$\beta$};
			\path[name path=post] (0, 3) -- (10, 3);

			\draw[name intersections={of= front1 and beta}] (intersection-1) -- +(0, -1.5) node[anchor=north east] {$\xi(t_1, \beta)$} ;
			\path[name intersections={of= front1 and post}] (intersection-1) node[anchor=south west]  {$t=t_1$}; 

			\draw[name intersections={of= front2 and beta}] (intersection-1) -- +(0, -1.5) node[anchor=north east] {$\xi(t_2, \beta)$} ;
			\path[name intersections={of= front2 and post}] (intersection-1) node[anchor=south west]  {$t=t_2$}; 

			\draw[->](0, 0) -- (10, 0);
			\draw[->](0, 0) -- (0, 4);

		\end{tikzpicture}
	\end{center}
\vspace{-0.5cm}
	\caption{\textit{A cartoon for the formation of the discontinuity. Here we choose $ t_1<t_2 $ and $ \xi(t,\beta),t=t_1,t_2 $ are the level sets. Theorem \ref{thm:discont} proves that  when Assumptions \ref{assum:initcond-basic} and \ref{assum:initcond-steep} are satisfied, then the distance $ |\xi(t, \beta)-h^*(t)| $  converges to 0 exponentially fast.}}
\end{figure}
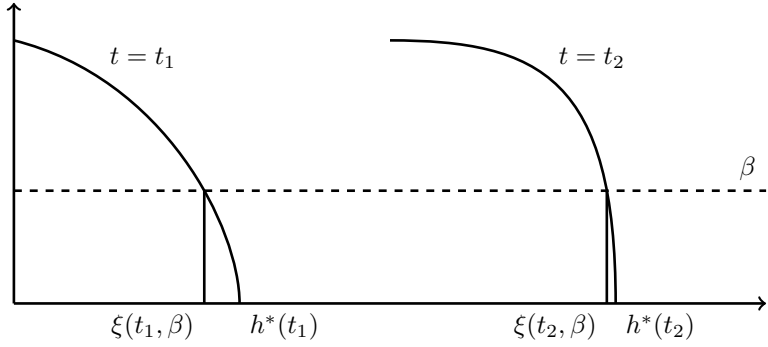
In particular, the profile $u(t, x)$ forms a discontinuity near the boundary point $h^*(t)$ as $t\to+\infty$. By considering  discontinuous integrated solutions, we are able to estimate the size of the jump for non increasing profiles, which leads to an estimate of the asymptotic speed.
\begin{prop}[Asymptotic jump near the separatrix]\label{prop:jump}
	Let $u_0$ be a non increasing function satisfying $u_0(-\infty)\leq 1$, $u_0(0)>0$ and $u_0(x)=0$ for $x> 0$. Then  
	\begin{align}
		\liminf_{t\to+\infty}u(t, h^*(t))&\geq \frac{2}{2+\hat\chi},\label{eq:jumpsize} \\
		\liminf_{t\to+\infty}\frac{\dd}{\dd t}h^*(t)&\geq \frac{\sigma\hat\chi}{2+\hat\chi}, \label{eq:charspeed-min}
	\end{align}
	where $\hat\chi=\frac{\chi}{\sigma^2}$.
\end{prop}

We finally turn to traveling wave solutions $ u(t,x)=U(x-ct) $, which are self-similar profiles traveling at a constant speed.
\begin{defn}[Traveling wave solution]\label{def:TW}
	A {\it traveling wave} is a positive solution  $u(t,x)$ to \eqref{eq:main} such that there exists a function $U\in L^\infty(\mathbb R)$ and a speed $c\in\mathbb R$ such that $u(t, x)=U(x-ct)$ for a.e. $(t, x)\in\mathbb R^2$. By convention, we also require that $U$ has the following behavior at $\pm\infty$:
	\[ \lim_{z\to - \infty}U(z) =1,\quad \lim_{z\to \infty} U(z)=0. \]
	The function $U$ is the {\it profile} of the traveling wave.
\end{defn}

Under a technical assumption on $\hat\chi=\frac{\chi}{\sigma^2}$, we can prove the existence of sharp traveling waves which present a jump at the vanishing point. 
\begin{assumption}[Bounds on $\hat\chi$]\label{as:hatchi}
	Let $\chi>0$ and $\sigma>0$ be given and define  $\hat\chi :=\frac{\chi}{\sigma^2}$. We assume that $0<\hat\chi<\bar\chi$, where $\bar\chi$ is the unique root of the function 
	\begin{equation*}
		\hat\chi\mapsto \ln\left(\frac{2-\hat\chi}{\hat\chi}\right)+\frac{2}{2+\hat\chi}\left(\frac{\hat\chi}{2}\ln\left(\frac{\hat\chi}{2}\right)+1-\frac{\hat\chi}{2}\right)
	\end{equation*}
	given by Lemma \ref{lem:B1}.
%
\end{assumption}
\begin{rem}
	It follows from Lemma \ref{lem:B1} that $\hat\chi=1$ satisfies Assumption 3. Actually, numerical evidence suggest that $\bar\chi\approx 1.045$.
\end{rem}
\begin{thm}[Existence of a sharp discontinuous traveling wave]\label{thm:sharp-TW}
	Let Assumption \ref{as:hatchi} be satisfied. There exists a traveling wave $u(t,x)=U(x-ct)$ traveling at speed 
	\begin{equation*}
		c\in\left(\frac{\sigma\hat\chi}{2+\hat\chi},\frac{\sigma\hat\chi}{2}\right),  
	\end{equation*}
	where $\hat\chi=\frac{\chi}{\sigma^2}$

	Moreover, the profile $U$ satisfies the following properties (up to a shift in space): 
	\begin{enumerate}[label={\rm (\roman*)}]
		\item $U$ is {\em sharp} in the sense that $U(x)=0$ for all $x\geq 0$; moreover, $U$ has a discontinuity at $x=0$ with $ U(0^-)\geq \frac{2}{2+\hat\chi} $.
		\item $U$ is continuously differentiable  and strictly decreasing on $(-\infty, 0]$, and satisfies
			\[ -c\,U'-\chi(UP')' =U(1-U)  \] 
			pointwise on $(-\infty, 0)$, where $P(z):=(\rho\star U)(z)$.
	\end{enumerate}
\end{thm}

Finally, we show that continuous traveling waves cannot be sharp, {\it i.e.} are necessarily positive on $\mathbb R$. 
\begin{prop}[Smooth traveling waves]\label{prop:smooth-TW}
	Let  $U(x)$ be  the profile of a traveling wave solution to \eqref{eq:main} and assume that $U$ is continuous. Then  $U\in C^1(\mathbb R)$, $U$ is strictly positive and we have the estimate:
	\begin{equation}\label{eq:smooth-speed}
		 -\chi(\rho_x\star U)(x)<c \text{ for all }x\in \mathbb R.
	\end{equation}
\end{prop}
In particular, by Theorem \ref{thm:discont}, any solutions starting from an initial condition satisfying Assumption \ref{assum:initcond-steep} may never catch such a traveling wave.

\section{Numerical Simulations}
We first describe the numerical framework of this study.
\begin{itemize}
	\item The parameters $\sigma $ and $\chi $ are fixed to 1, $\sigma=1$ and $\chi=1$.
	\item We are given a bounded interval $ [-L,L] $ and an initial distribution of $ \phi \in C([-L,L]) $;
	\item We solve numerically the following PDE using the upwind scheme ($p$ being given)
	\begin{equation}\label{5.1}
	\begin{cases}
	\partial _{t}u(t,x) -\partial_x\bigl(u(t,x)\partial_x p(t,x)\bigr)=u(t,x)(1-u(t,x)),\\
	\nabla p(t,x) \cdot \nu =0\\
	u(0,x)=\phi(x),
	\end{cases}\;\;t>0,\;x\in [-L,L].
	\end{equation}
	\item The pressure $ p $ is defined as 
	\begin{equation}\label{5.2}
	\begin{aligned}
	p(t,x)=(I-\Delta)_{\mathcal{N}}^{-1} u(t,x), \quad t>0,x\in [-L,L],
	\end{aligned}
	\end{equation}
	where $ (I-\Delta)_{\mathcal{N}}^{-1} $ is the Laplacian operator with Neumann boundary condition. Due to the Neumann boundary condition of the pressure $ p $, we do not need boundary condition on $ u $ (see \cite{Nadin2008,Fu2019}).
\end{itemize}
Our numerical scheme reads as follows
\begin{equation*}
\begin{aligned}
&\frac{u^{n+1}_i-u^n_i}{\Delta t} +\, \frac{1}{\Delta x}\bigg(G(u_{i+1}^n,u_{i}^n)-G(u_{i}^n,u_{i-1}^n)\bigg)= u_i^n(1-u_i^n),\\
&i=1,2,\ldots,M,\ n=0,1,2,\ldots\\
&u_0=1,\;u_{M+1}=0,
\end{aligned}
\end{equation*}
with $ G(u_{i+1}^n,u_{i}^n) $ defined as
\begin{equation*}
G(u_{i+1}^n,u_{i}^n)=(v_{i+\frac{1}{2}}^n)^+u_{i}^n -(v_{i+\frac{1}{2}}^n)^-u_{i+1}^n=\begin{cases}
v_{i+\frac{1}{2}}^n u_{i}^n, & v_{i+\frac{1}{2}}^n\geq 0,\\
v_{i+\frac{1}{2}}^n u_{i+1}^n, &v_{i+\frac{1}{2}}^n<0,
\end{cases}\quad i=1,\ldots,M.
\end{equation*}
Moreover the velocity $ v $ is given by
\begin{equation*}
v_{i+\frac{1}{2}}^n=-\dfrac{p_{i+1}^n-p_i^n}{\Delta x},\ i=0,1,2,\cdots,M,
\end{equation*}
where from \eqref{5.2} we define 
\begin{equation*}
P^n:=(I- A)^{-1}U^n,\quad P^n = \big(p_i^n\big)_{M\times 1}\quad
U^n = \big(u_i^n\big)_{M\times 1}.
\end{equation*}	
where $ A=(a_{i,j})_{M\times M} $ is the usual linear diffusion matrix with Neumann boundary condition. Therefore, by Neumann boundary condition $ p_0=p_1 $ and $ p_{M+1}=p_M $, when $ i=1,M $ we have 
\begin{equation*}
\begin{aligned}
G(u_{1}^n,u_{0}^n)=0,\\
G(u_{M+1}^n,u_{M}^n)=0,
\end{aligned}
\end{equation*}
which gives
\begin{align*}
u^{n+1}_1&=u^n_1 -d\, \frac{\Delta t}{\Delta x}G(u_{2}^n,u_{1}^n)+\Delta t \, f(u_1^n),\\
u^{n+1}_M&=u^n_M+d\, \frac{\Delta t}{\Delta x}G(u_{M}^n,u_{M-1}^n)+\Delta t \, f(u_M^n).
\end{align*}
Owing to the boundary condition, we have the conservation of mass for Equation \eqref{5.1} when the reaction term equals zero.

\subsection{Formation of a discontinuity}

In this part, we use numerical simulations to verify the theoretical predictions in the previous sections.
Firstly, we choose the initial value $ \phi \in C^1([-L,L]) $ as follows 
\begin{equation}\label{5.3}
\phi(x)= \frac{(x-x_0)^2}{(L+x_0)^2} \mathbbm{1}_{ [-L, x_0] }(x),\quad L=20,\;x_0=-15.
\end{equation}
Notice that this initial condition satisfies Assumptions \ref{assum:initcond-basic} and \ref{assum:initcond-steep}. Due to Theorem \ref{thm:discont}, we should observe the formation of a discontinuity in space for large time.

We plot the evolution of the solution $ u(t,x) $ starting from $ u(0,x) =\phi(x) $ in Figure \ref{FIG1}.
\begin{figure}[H]
	\begin{center}
		\includegraphics[width=0.45\textwidth]{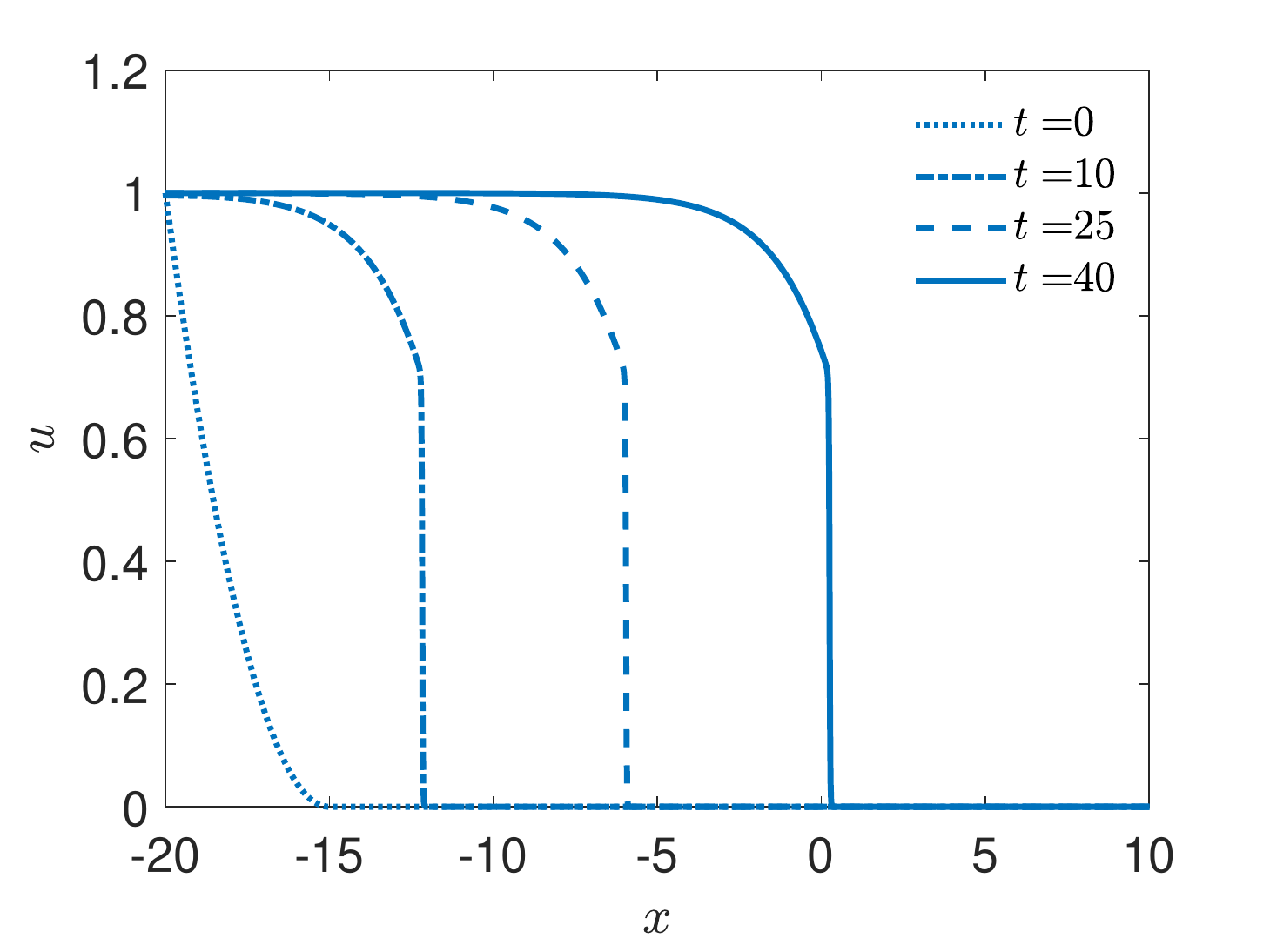}
		\vspace{-0.5cm}
	\end{center}
	\caption{\textit{We plot the propagation of the traveling waves under system \eqref{5.1} with the initial value \eqref{5.3}. We plot the propagation profile at $ t=0,10,25,40 $ (resp. dashed lines, dotted-dashed lines, dotted lines and solid lines). }}
	\label{FIG1}
\end{figure}
We observe that the jump is formed for large time and the height of the jump is greater than $ 2/3$ which is in accordance with Theorem \ref{thm:sharp-TW}.

Next, we study the propagation speed of different level sets, namely,
\begin{equation*}
t\longmapsto \xi(t, \beta)+L,
\end{equation*}
where $ \xi(t, \beta):=\sup\{x\in\mathbb R\,|\, u(t, x)= \beta\} $ and $ \beta = 0,0.2,2/3, 0.8  $. Note that the case  $ \beta =0 $ corresponds to the rightmost characteristic.

We compute the propagation speed in the following way:
for different $ \beta\in [0,1] $, we choose $ t_1=15 $ and $ t_2=40 $ where the propagation speed is almost stable after $ t=t_1 $. Thus we can compute the mean propagation speed as follows
\begin{equation}\label{eq:calculation-v}
\text{Propagation speed at level $ \beta $}= \frac{\xi(t_2,\beta)-\xi(t_1,\beta)}{t_2-t_1}.
\end{equation}

\begin{figure}[H]
	\begin{center}
		\includegraphics[width=0.45\textwidth]{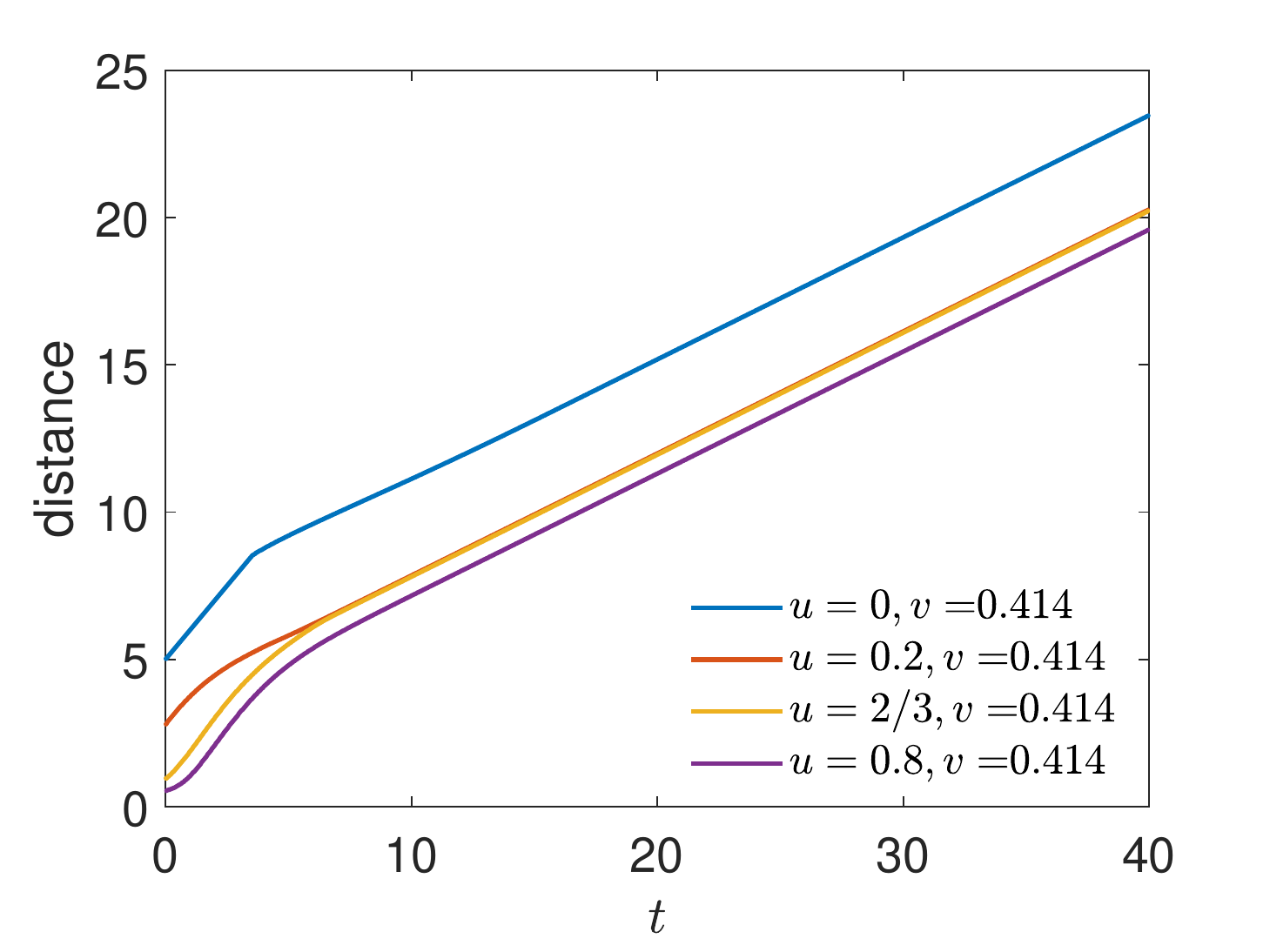}
		\vspace{-0.5cm}
	\end{center}
	\caption{\textit{We  plot the evolution of different level sets $ t\longmapsto \xi(t, \beta)+L $ under system \eqref{5.1}. Our initial distribution is taken as \eqref{5.3}. We plot the propagating speeds of the profile at $ \beta =0,0.2,2/3, 0.8 $. The x-axis represents the time and the y-axis is the relative distance $ \xi(t, \beta)+L $. The velocity is calculated by \eqref{eq:calculation-v} for $ t_1=15 $ and $ t_2=40 $.}}
	\label{FIG2}
\end{figure}

%


Next we want to check whether the solutions of system \eqref{5.2} starting from two different initial values  converge to the same discontinuous traveling wave solution. To that aim, given two different initial profiles $ \phi_1 $ and $ \phi_2 $ with $ \phi_1 \leq \phi_2 $ on $ [-L,L] $,
\begin{equation}\label{5.4}
\phi_1(x)= -\frac{x+15}{5} \mathbbm{1}_{ [-20, -15] }(x),\quad \phi_2(x)=\mathbbm{1}_{ [-20, -17.5] }(x) -\frac{x+15}{10} \mathbbm{1}_{ [-17.5, -15] }(x)
\end{equation}
We simulate the propagation of these two profiles in Figure \ref{FIG4}.
\begin{figure}[H]
	\begin{center}
		\includegraphics[width=0.45\textwidth]{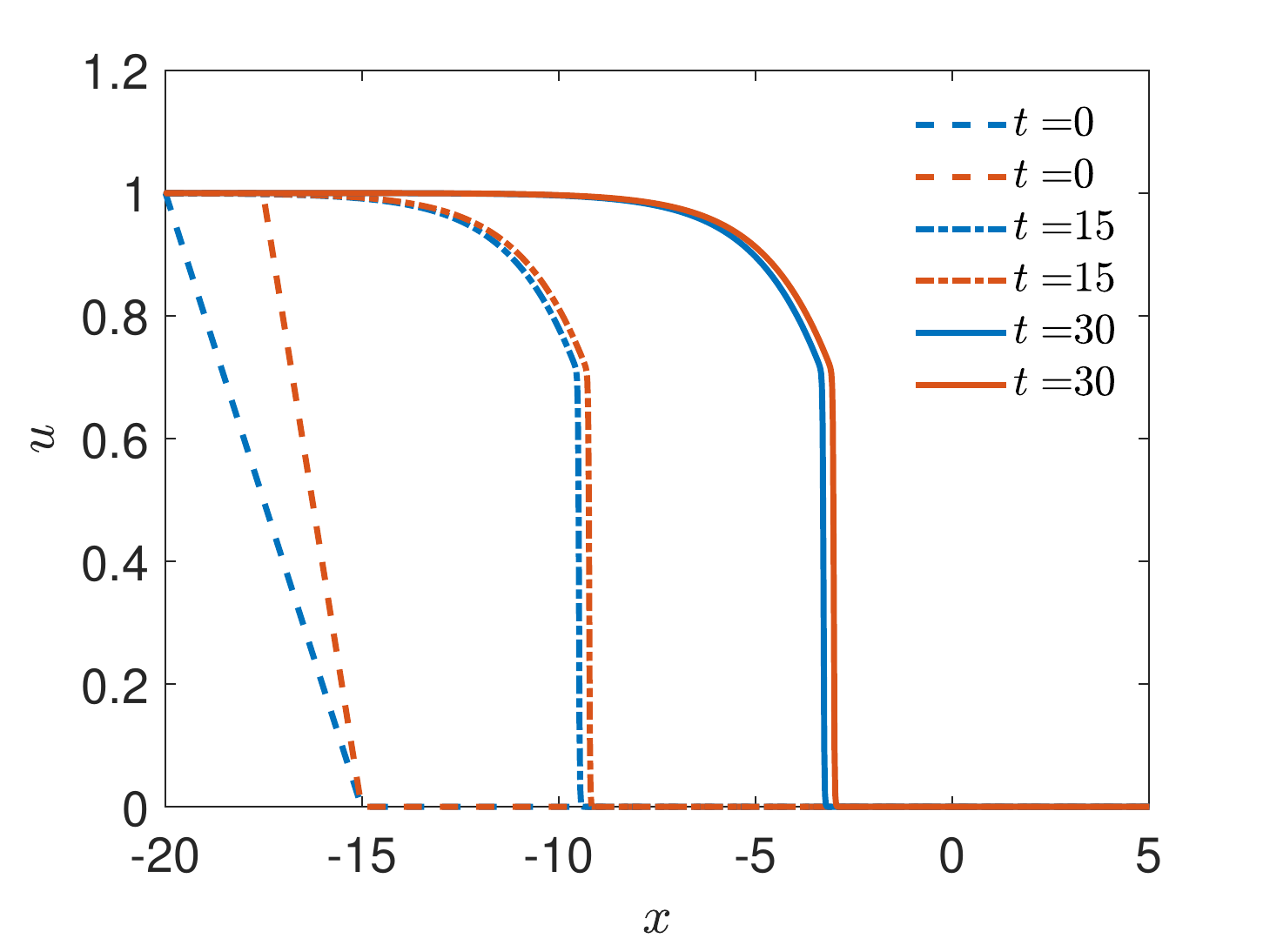}
		\vspace{-0.5cm}
	\end{center}
	\caption{\textit{We  plot the propagation of two profiles under system \eqref{5.1} with initial distributions are taken as \eqref{5.4}. The blue curves represent the profile with initial distribution $ \phi_1 $ while the red curves represent the profile with initial distribution $ \phi_2 $. We plot the propagation profiles at $ t=0,15 $ and $ 30 $ (resp. dashed lines, dotted-dashed lines and solid lines). The simulation shows that the two profiles converge to the same discontinuous traveling wave solution.  }}
	\label{FIG4}
\end{figure}
\subsection{Large speed traveling waves}
As we know for porous medium equation, the existence of large speed $ c>c_* $ traveling wave solutions is proved in \cite{dePablo1991} and it can be observed numerically by taking the exponentially decreasing function as initial value. In this part, instead of taking a compactly supported initial value, we set the initial value 
\begin{equation}\label{eq:5.5-initial-value}
\phi_{\alpha}(x) = \frac{1}{1+e^{\alpha(x-x_0)}},\quad x_0=-15,
\end{equation}
where $ \alpha\geq 1 $ is a parameter introduced to describe the decaying rate of the initial value. 

We compare the following three different scenarios with different parameters $ \alpha=1,2, 5 $ in the initial value \eqref{eq:5.5-initial-value}.
\begin{figure}[H]
	\begin{center}
		\includegraphics[width=0.3\textwidth]{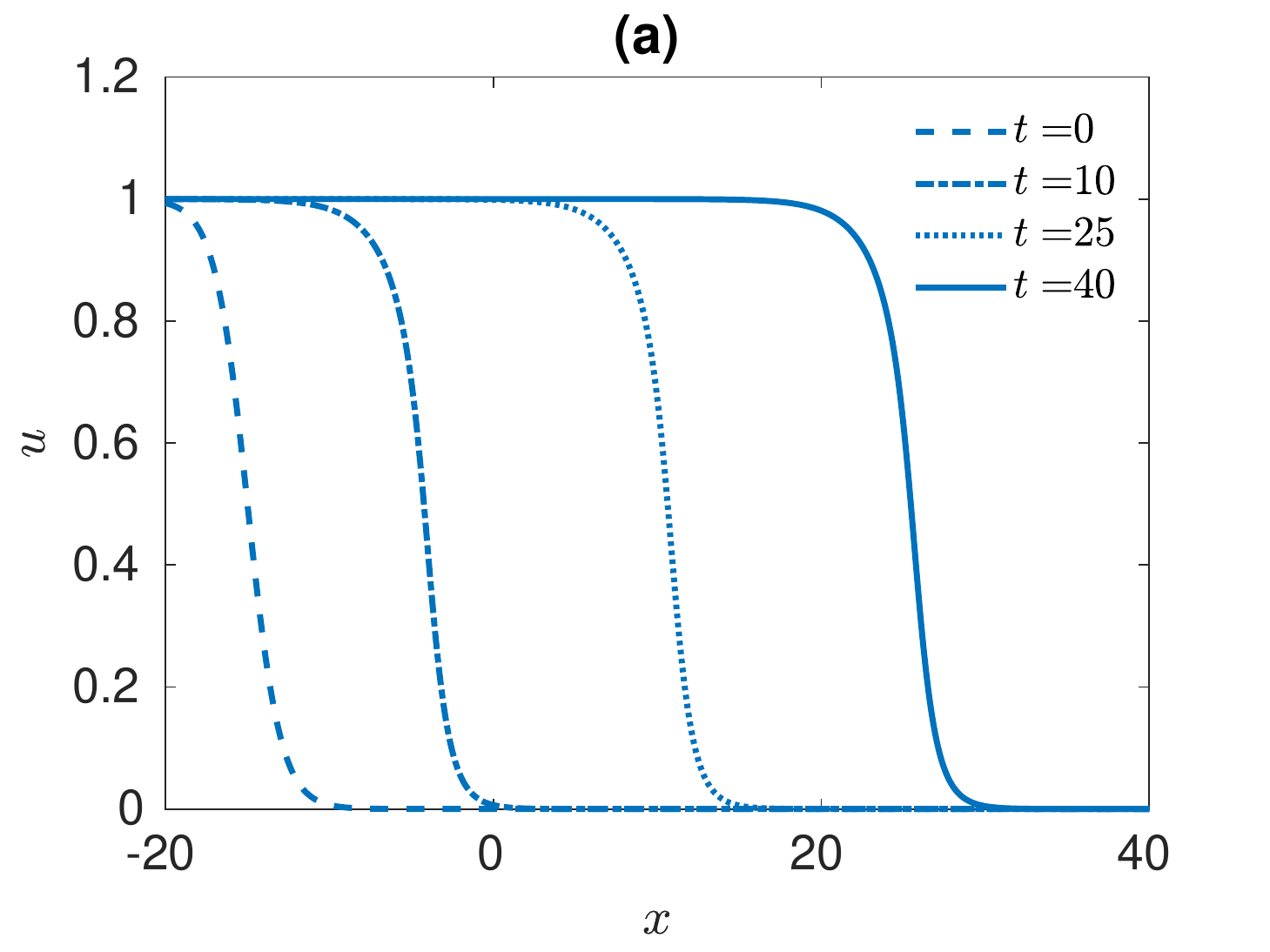}\includegraphics[width=0.3\textwidth]{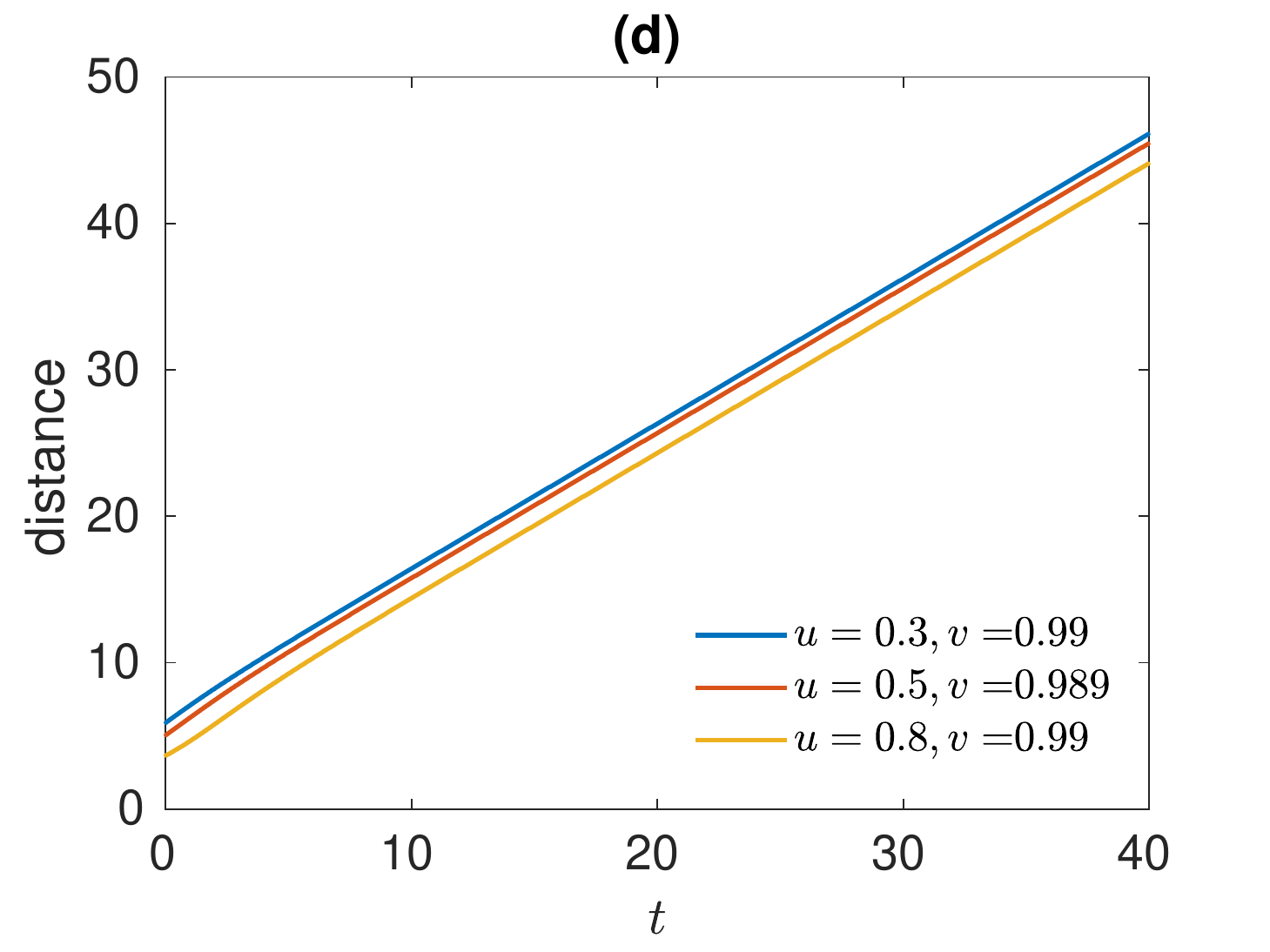}
		\includegraphics[width=0.3\textwidth]{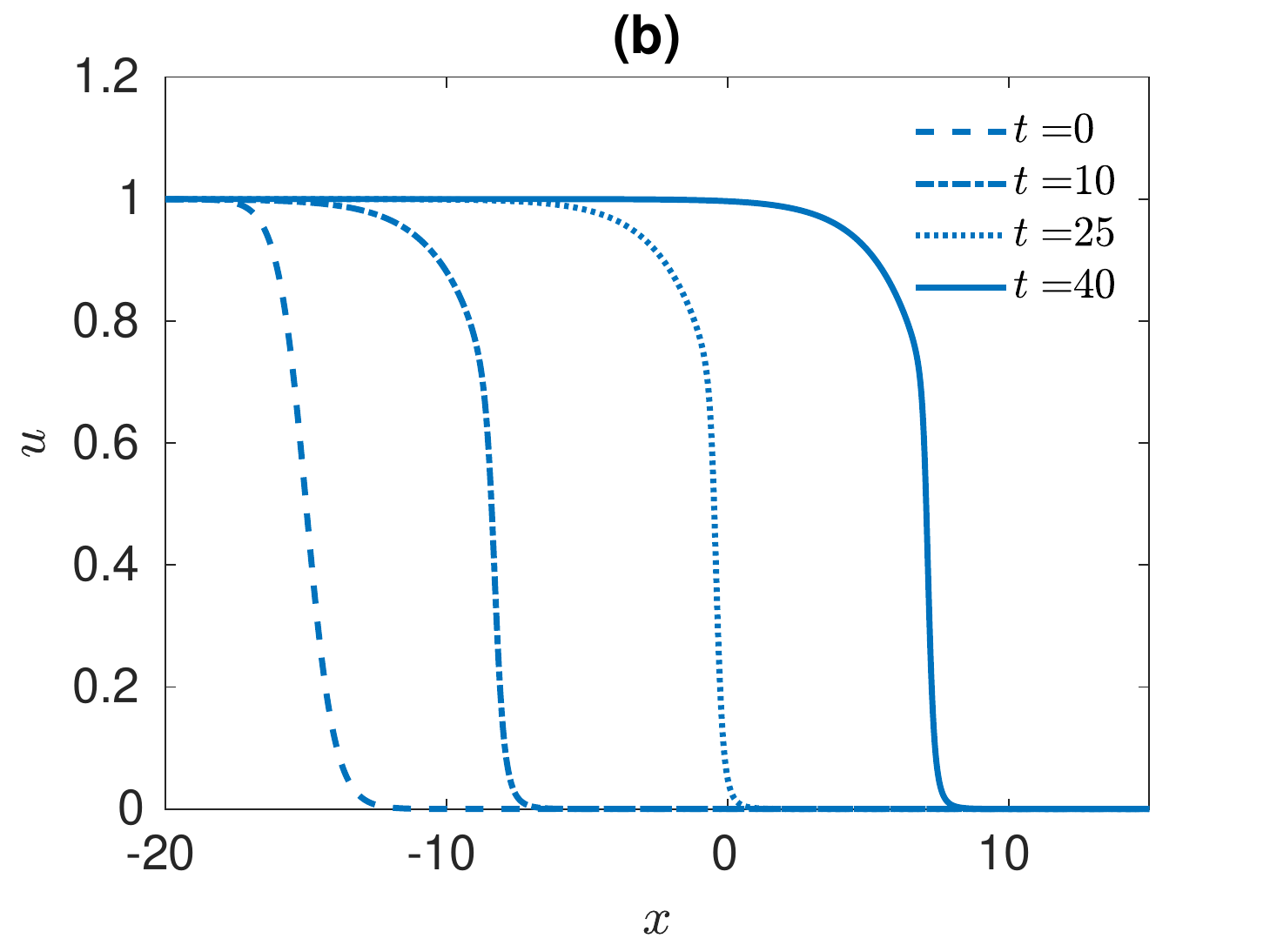}\includegraphics[width=0.3\textwidth]{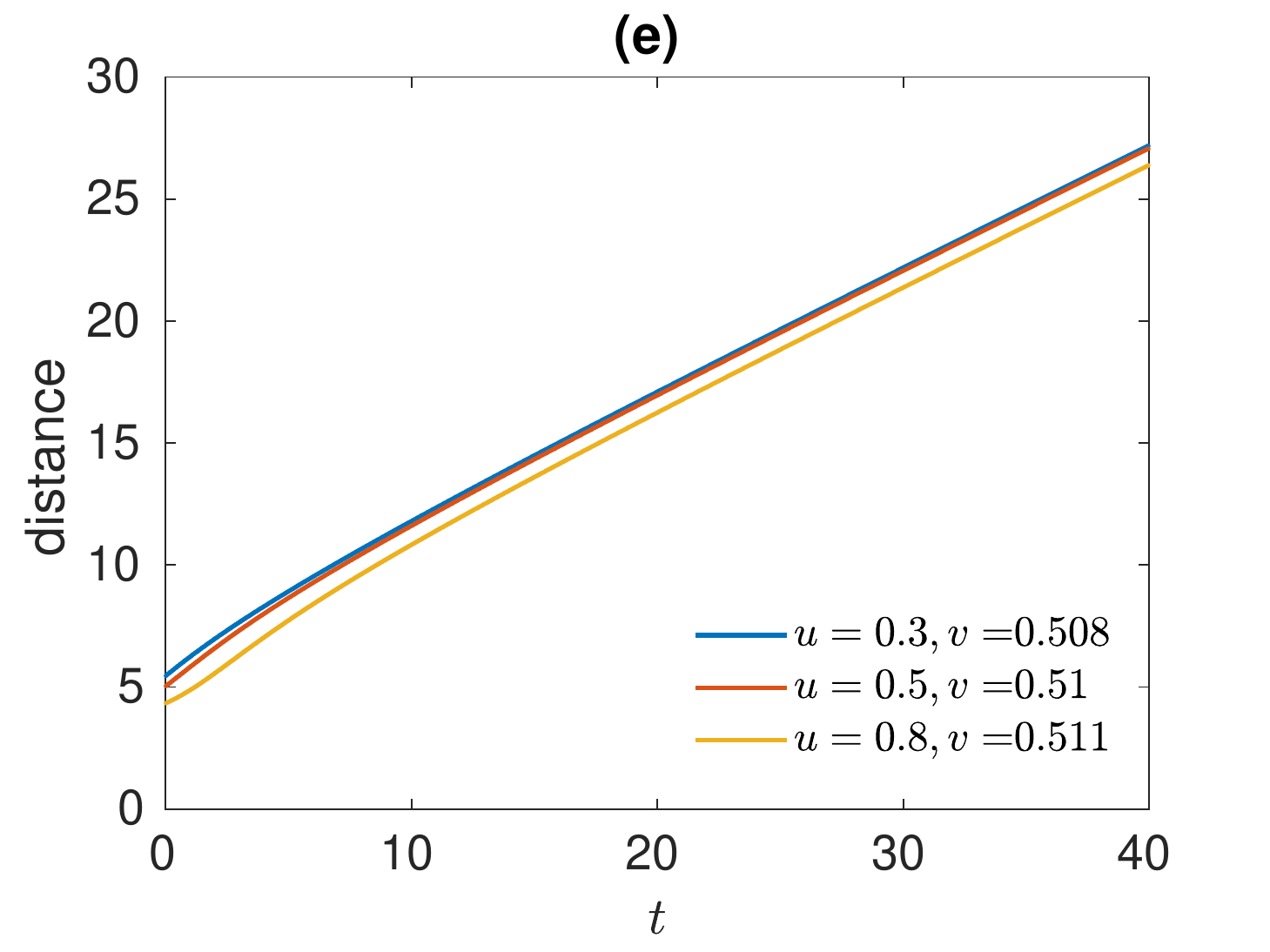}
		\includegraphics[width=0.3\textwidth]{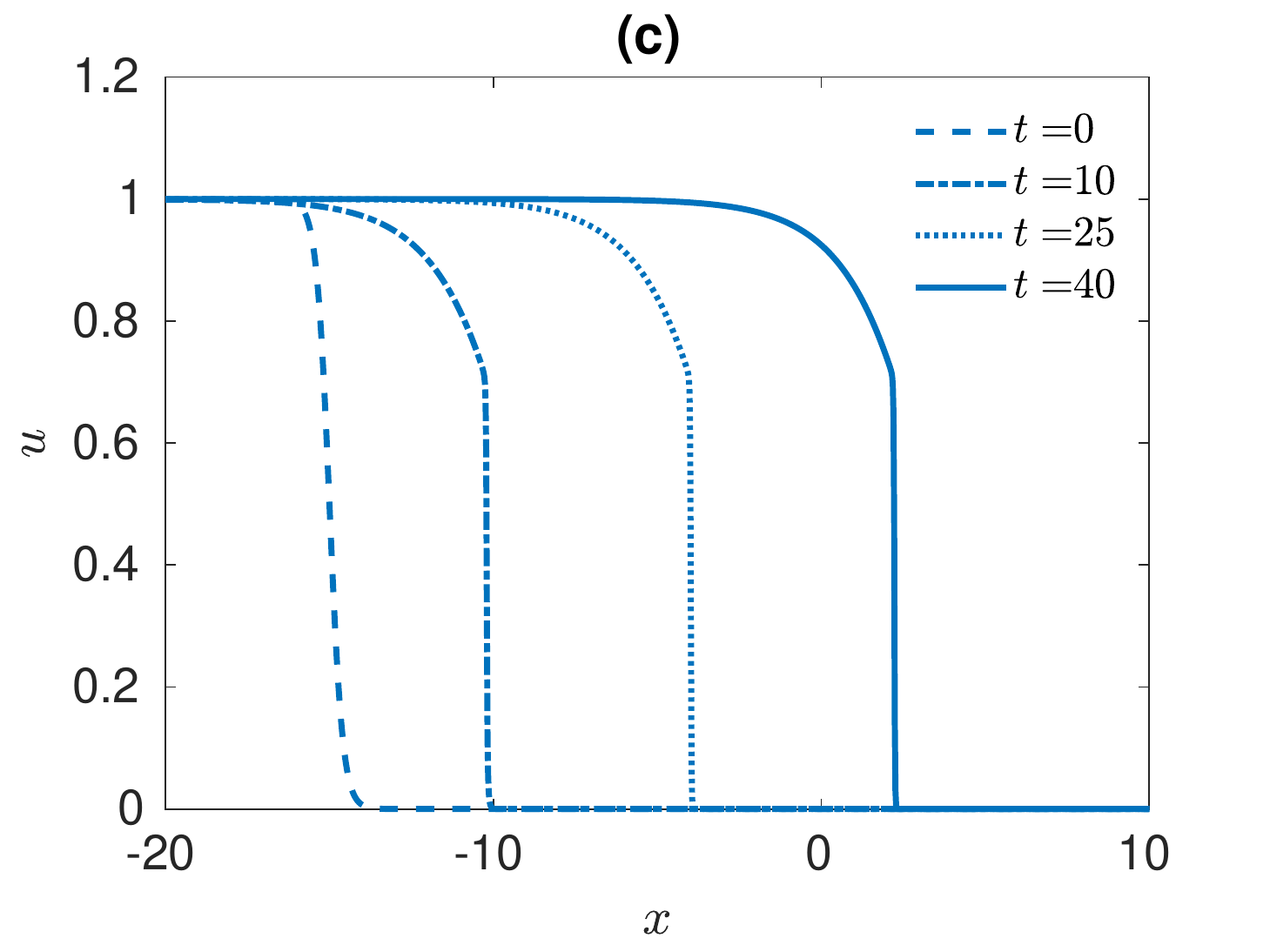}\includegraphics[width=0.3\textwidth]{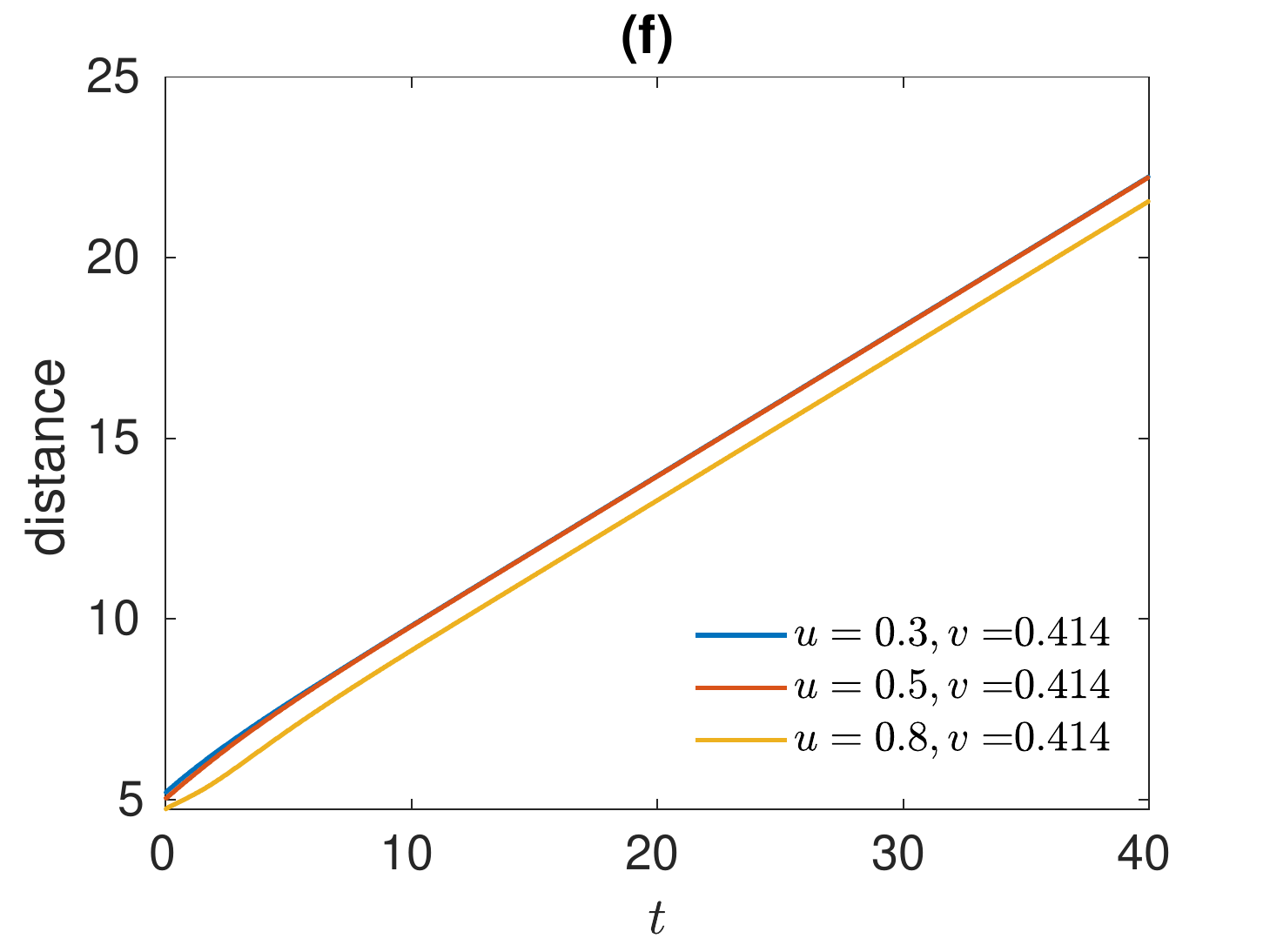}
	\end{center}
	\caption{\textit{We plot the propagation of the traveling waves under system \eqref{5.1} with the initial values \eqref{eq:5.5-initial-value} and  the corresponding evolution of different level sets $ t\longmapsto \xi(t, \beta)+L $. Figure (a) and (d) represent the evolution of the traveling wave and its level sets when $ \alpha=1 $. Figure (b) and (e) correspond to the case when $ \alpha=2 $. Figure (c) and (f) correspond to the case when $ \alpha=5 $. }}
	\label{FIG5}
\end{figure}
We observe the large speed traveling waves in Figure \ref{FIG5} when $ \alpha =1,2 $.
We note that as the parameter $ \alpha $ in \eqref{eq:5.5-initial-value} is increasing,  the propagation speed is decreasing and $ c \approx 1/\alpha $. When $ \alpha =5 $, the propagation of the traveling waves is similar to the case in Figure \ref{FIG1} in which we started from the compactly supported initial value. In other word, we can observe	 the formation of discontinuity and the critical speed $ c_*\approx 0.414  $ is reached. 

\subsection{Comparison with porous medium equations: the vanishing jump}
\label{Section3.3}
In this part, we compare the non-local advection model with the porous medium equation by introducing a new parameter $ \sigma $ 
\begin{equation}\label{eq:rho-epsilon}
p(t,x)=(I-\sigma^2 \Delta)_{\mathcal{N}}^{-1} u(t,x)
\end{equation}
Thus if $ \sigma \to 0 $ then formally we have $ p(t,x) \to u(t,x) $. Thus, the first equation of \eqref{5.1} becomes
\[ u_t - \frac{1}{2}(u^2)_{xx} =u(1-u) ,\]
which is the classical porous medium equation. It is well-known that this equation has the explicit traveling wave solution $ U(z)=(1-e^{z/\sqrt{2}})_+ $
with critical speed $ c_* =1/\sqrt{2} $.

We are consider the transition from the discontinuous traveling wave solution to the continuous sharp-type traveling wave solution by letting $ \sigma \to 0 $. Moreover, we want to see if the critical traveling speed of the discontinuous wavefront $ c(\sigma)$ converges to $  c_*=1/\sqrt{2}\approx 0.707 $ as $ \sigma \to 0 $. Our initial value is taken as $ 1/(1+{\rm exp}(5*(x+15))), x\in [-20,20] $ in \eqref{eq:5.5-initial-value}. We compare the following three different scenarios with different parameters $ \sigma^2=0.5,0.1,0.01 $ in kernel \eqref{eq:rho-epsilon}.
\begin{figure}[H]
	\begin{center}
		\includegraphics[width=0.3\textwidth]{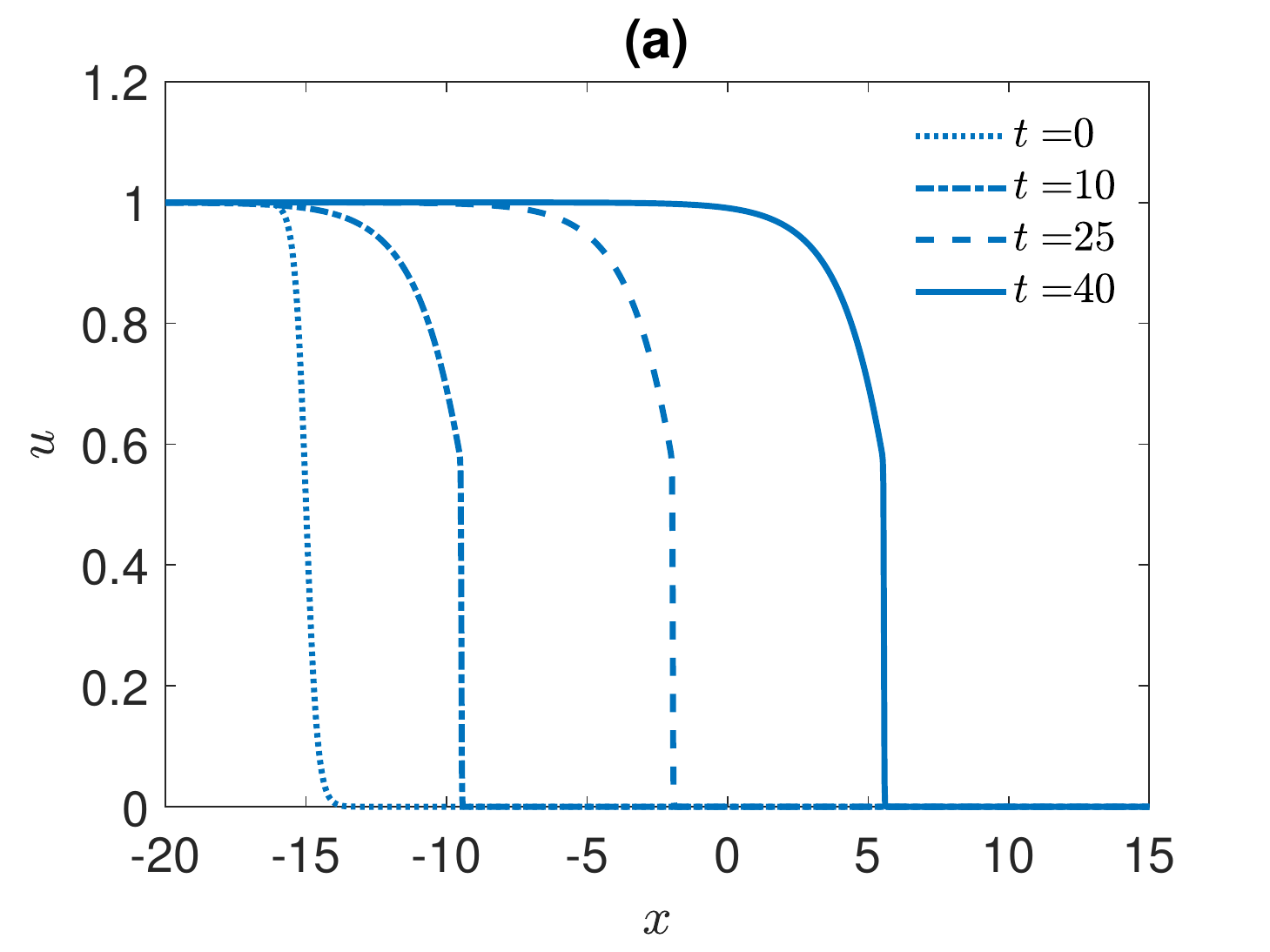}\includegraphics[width=0.3\textwidth]{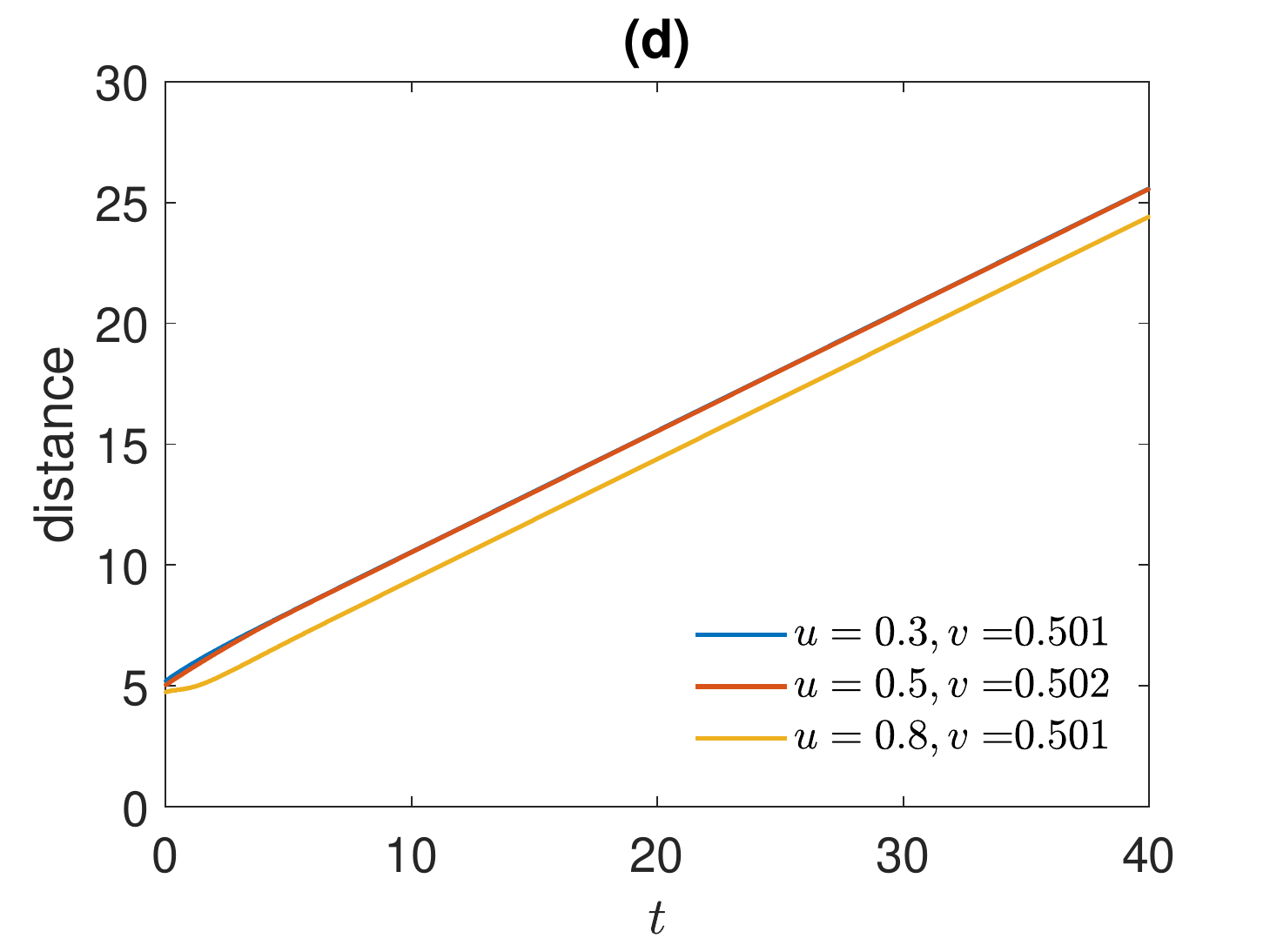}
		\includegraphics[width=0.3\textwidth]{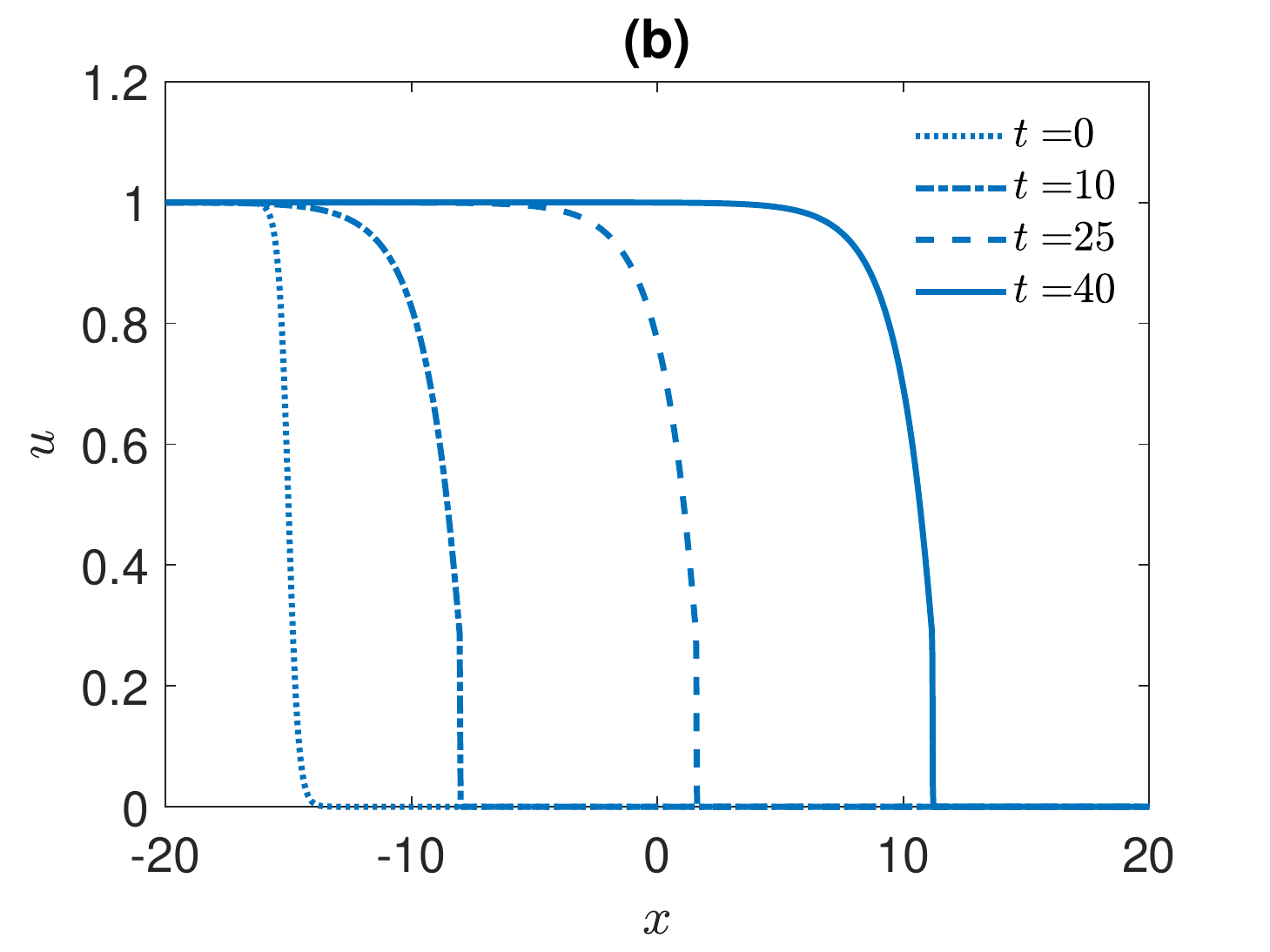}\includegraphics[width=0.3\textwidth]{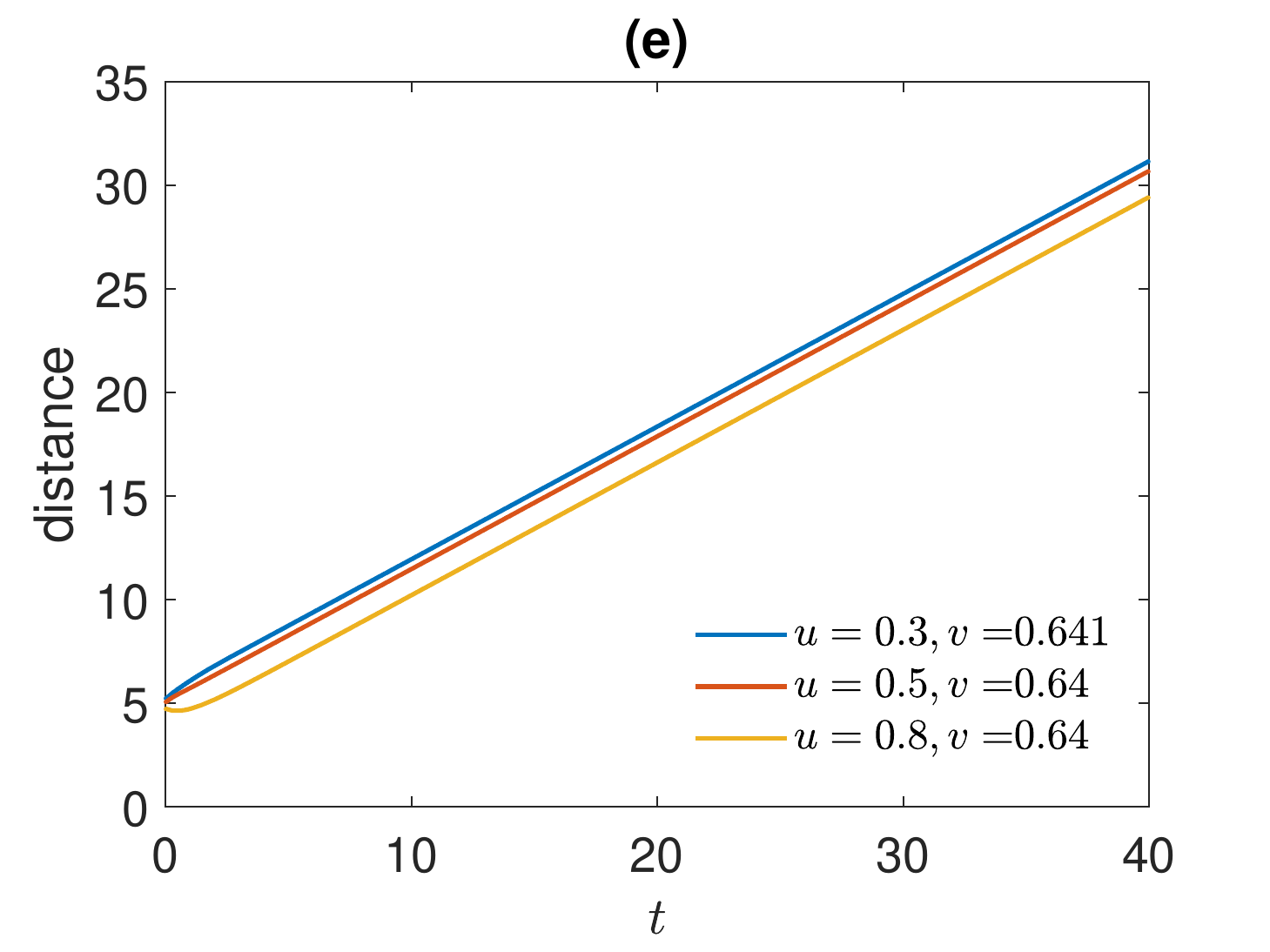}
		\includegraphics[width=0.3\textwidth]{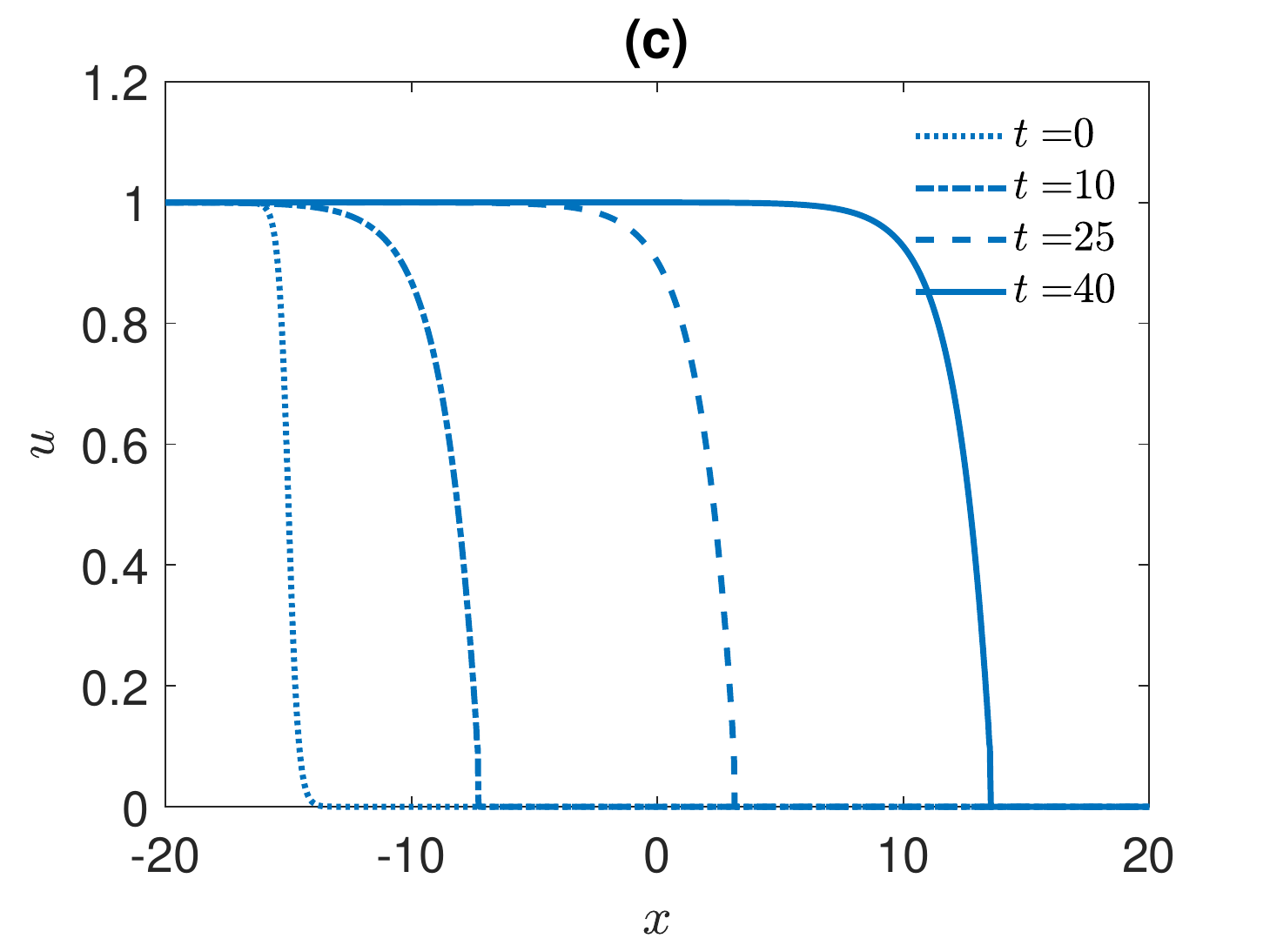}\includegraphics[width=0.3\textwidth]{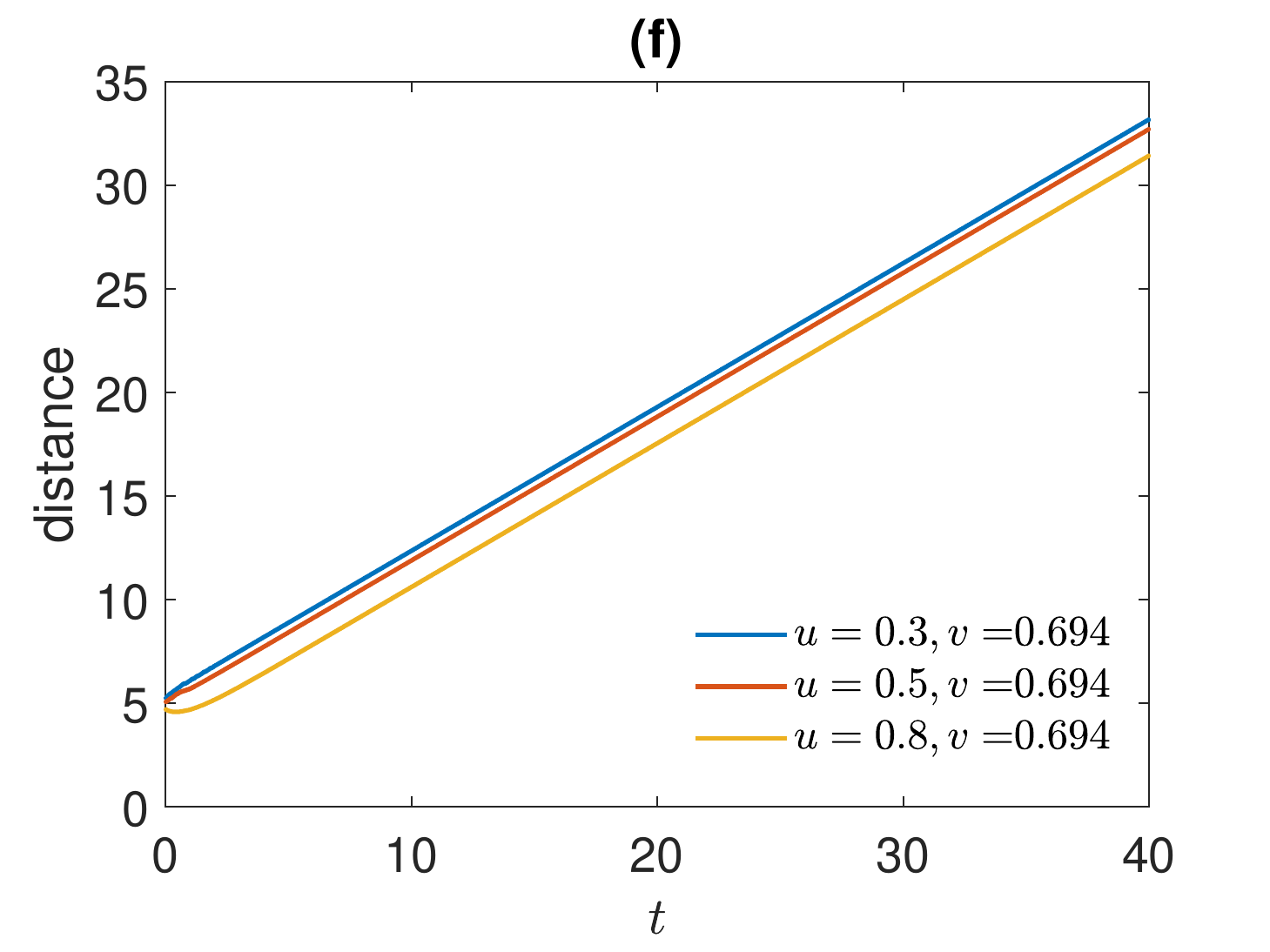}		
	\end{center}
	\caption{\textit{We plot the propagation of the traveling waves for system \eqref{5.1} with the kernel \eqref{eq:rho-epsilon} and  the corresponding evolution of different level sets $ t\longmapsto \xi(t, \beta)+L $.  Figure (a) and (d) represent the evolution of the traveling wave and its level sets when $ \sigma^2=0.5 $. Figure (b) and (e) correspond to the case when $ \sigma^2=0.1 $. Figure (c) and (f) correspond to the case when $ \sigma^2=0.01 $. Our initial value is taken as in \eqref{eq:5.5-initial-value} with $ \alpha =5 $. }}
	\label{FIG6}
\end{figure}

In Figure \ref{FIG6} we can observe that as $ \sigma \to 0 $ in the kernel, the discontinuous jump is gradually vanishing from (a) to (c). Moreover, the critical speed $ c(\sigma) $ is increasing as $ \sigma \to 0 $ and is approaching the critical speed $ c_*=1/\sqrt{2}\approx0.707 $ for the porous medium case.

\section{Properties of the time-dependent solutions}
\label{sec:spreading}

\subsection{The separatrix}

In this section we study the qualitative properties of solutions to \eqref{eq:main} starting from an initial condition supported in $(-\infty, 0]$. 

\begin{prop}[The separatrix]\label{prop:rightmost}
	Let $u$ be a solution integrated along the characteristics to \eqref{eq:main}, starting from $u_0(x)$ satisfying Assumption \ref{assum:initcond-basic}. Let $h^*(t):=h(t,0)$ be the separatrix  (as in Proposition \ref{prop:separatrix}). Then $h^*(t)$ stays at the rightmost boundary of the support of $u(t, \cdot)$, i.e.
	\begin{enumerate}[label={\rm (\roman*)}]
		\item \label{item:ultimatelytrivial2}
			we have
			\begin{equation}
				 u(t, x)=0 \text{ for all } x\geq h^*(t).
			\end{equation}
		\item \label{item:hboundary2}
			for each $t>0$ there exists $\delta>0$ such that  
			\begin{equation}
				u(t, x)>0 \text{ for all } x\in (h^*(t)-\delta, h^*(t)).
			\end{equation}
	\end{enumerate}
\end{prop}
\begin{proof}
	By definition the characteristics are well-defined by \eqref{eq:characteristics} as the flow of an ODE. In particular, if $x\geq h^*(t)= h(t, 0) $ there exists $x_0\geq 0$ such that $x=h(t, x_0)$. Since $u_0(x_0)=0$ and in view of \eqref{eq:integrated}, we have indeed $u(t, x)=0$. This proves Item \ref{item:ultimatelytrivial2}

	By  Assumption \ref{assum:initcond-basic}, there exists $\delta_0>0$ such that $u_0(x)>0$ for $x\in (-\delta_0, 0)$.  We remark that
	\begin{align*}
		\frac{\dd }{\dd t}u(t, h(t, x))&=\hat\chi \,u(t, h(t,x))\big((\rho*u)(t, h(t,x))-u(t, h(t,x))\big)+u(t, h(t, x))\big(1-u(t, h(t,x))\big)\\
		&\geq u(t,h(t, x))\big(1-(1+\hat\chi)u(t,h(t,x))\big).
	\end{align*}
	By comparison with the solution to the ODE $v'(t)=v(t)(1-(1+\hat\chi)v(t))$ starting from $v(t=0)=u_0(x)>0$, we deduce that  $u(t, x)\geq v(t)>0$ for each $x\in (h(t, -\delta_0), h^*(t))$. Since $h(t, -\delta_0)<h(t,0)=h^*(t)$, this proves Item \ref{item:hboundary2}.
\end{proof}

Next we investigate the propagation of $u$. 
\begin{prop}[$u$ is propagating]\label{prop:propagating}
	Let $u_0$ satisfy Assumption \ref{assum:initcond-basic} and let $u$ be the solution integrated along the characteristics to \eqref{eq:main} starting from $u(t=0,x)=u_0(x)$. Then $u $ is propagating to the right, i.e.
	\begin{equation}\label{eq:hincreasing}
		\frac{\dd}{\dd t}h^*(t)>0.
	\end{equation}
	Moreover, we have the estimate:
	\begin{equation}\label{eq:maxspeed}
		\frac{\dd }{\dd t}h^*(t)\leq \frac{\chi}{2\sigma}.
	\end{equation}
\end{prop}
\begin{proof}
	We have the following estimates:
	\begin{align*}
		\frac{\dd}{\dd t} h^*(t)&=-\chi(\rho_x*u)(t, h^*(t))\\
		&=-\chi\int_{-\infty}^{+\infty}\rho_x(y)u(t, h^*(t)-y)\dd y\\
		&=\chi\int_{-\infty}^{+\infty}\frac{\text{sign}(y)}{2\sigma^2}e^{-\frac{|y|}{\sigma}}u(t, h^*(t)-y)\dd y\\
		&=\frac{\chi}{\sigma}\int_0^{+\infty}\rho(y)u(t, h^*(t)-y)\dd y\\
		&>0, 
	\end{align*}
	since $u(t, x)=0$ for all $x> h^*(t)$. \eqref{eq:hincreasing} is proved.

	Then, since $0\leq u\leq 1$, we have
	\begin{align*}
		\frac{\dd}{\dd t} h^*(t)&=\frac{\chi}{\sigma}\int_0^{+\infty}\rho(y)u(t, h^*(t)-y)\dd y\\
		&\leq \frac{\chi}{\sigma}\int_0^{+\infty}\rho(y)\dd y = \frac{\chi}{2\sigma},
	\end{align*}
	which proves \eqref{eq:maxspeed}.
\end{proof}
These first two propositions together yield a proof of Proposition \ref{prop:separatrix}.
\begin{proof}[Proof of Proposition \ref{prop:separatrix}]
	Items \ref{item:ultimatelytrivial} and \ref{item:hboundary} have been proved in Proposition \ref{prop:rightmost}, and the propagating property follows from Proposition \ref{prop:propagating}.
\end{proof}

We continue with a technical lemma that will be used in the proof of Theorem \ref{thm:discont}. 
\begin{lem}[Divergence speed near the separatrix]\label{lem:div-sep}
	Let $u_0(x)$ satisfy Assumptions \ref{assum:initcond-basic} and \ref{assum:initcond-steep} and $u(t,x)$ be the corresponding solution to \eqref{eq:main}. Let $h(t,x)$ be the characteristic flow of $u$ and  $h^*(t)$ be the separatrix of $u$, as defined in Proposition \ref{prop:separatrix}. For all $t\geq 0$ and $x<0 $ we have
		\begin{equation}\label{eq:propgrowth-div}
			\frac{\dd}{\dd t}(h^*(t)-h(t, x))\leq\chi\,(h^*(t)-h(t, x))\sup_{y\in(h(t,x), h^*(t))}u(t,y).
		\end{equation}
\end{lem}
\begin{proof}
		Recall that, by Proposition \ref{prop:rightmost}, $u(t, x)=0 $ for each $x\geq h^*(t)$. 
		For $x<0$, we notice that:
		\begin{align*}
			\frac{\dd}{\dd t}\big(h^*(t)-h(t, x)\big)&=-\chi(\rho_x\star u)(t, h^*(t)) + \chi(\rho_x\star u)(h(t,x)) \\
			&=\chi\int_{\mathbb R}\big(\rho_x(h(t,x)-y)-\rho_x(h^*(t)-y)\big)u(t, y)\dd y\\
			&=\chi\int_{-\infty}^{h(t,x)}\big(\rho_x(h(t,x)-y)-\rho_x(h^*(t)-y)\big)u(t, y)\dd y\\
			&\quad +\chi\int_{h(t,x)}^{h^*(t)}\big(\rho_x(h(t,x)-y)-\rho_x(h^*(t)-y)\big)u(t, y)\dd y.
		\end{align*}
		Therefore, 
		\begin{align*}
			\notag \frac{\dd}{\dd t}(h^*(t)-h(t, x))&\leq\chi\int_{-\infty}^{h(t, x)}\big(\rho_x(h(t,x)-y)-\rho_x(h^*(t)-y)\big)u(t, y)\dd y  \\
			\notag&\quad +\chi(h^*(t)-h(t, x))\times\sup_{y\in(h(t,x), h^*(t))}u(t,y).
		\end{align*}
		Since $\rho_x(y)=-\frac{1}{2\sigma^2}\text{sign}(y)e^{-\frac{|y|}{\sigma}}$ is increasing on $(0, +\infty)$, we have  
		\begin{equation*}
			\rho_x(h(t,x)-y)-\rho_x(h^*(t)-y)\leq 0
		\end{equation*}
		for each $y\leq  h(t,x)$, which shows \eqref{eq:propgrowth-div}. Lemma \ref{lem:div-sep} is proved.
\end{proof}

\begin{prop}[Formation of a discontinuity] \label{prop:growth}
	Let $u_0(x)$ satisfy Assumptions \ref{assum:initcond-basic} and \ref{assum:initcond-steep} and $u(t,x)$ be the corresponding solution to \eqref{eq:main}. For all $\delta>0$ we have 
	\begin{equation}\label{eq:discont2}
		\limsup_{t\to+\infty}\sup_{x\in (h^*(t, x)-\delta, h^*(t))}u(t,x)\geq \frac{1}{1+\hat\chi+\alpha\chi}>0.
	\end{equation}
\end{prop}
\begin{proof}
	We divide the proof in 2 steps.

	\begin{stepping}
		\step We show that for all $\delta>0$, 
		\begin{equation}\label{eq:discont-sup}
			\sup_{t>0}\sup_{x\in (h^*(t)-\delta, h^*(t))}u(t,x)\geq \frac{1}{1+\hat\chi+\alpha\chi}.
		\end{equation}

		Assume by contradiction that there exists $\delta>0$  such that 
		\begin{equation}\label{eq:propgrowth-usmall}
			\forall t>0, \sup_{x\in (h^*(t)-\delta, h^*(t))}u(t,x)\leq \eta< \frac{1}{1+\hat\chi+\alpha\chi},
		\end{equation}
		where $\alpha\geq 1$ is the constant from Assumption \ref{assum:initcond-steep}. 
		
		We remark that the following inequality holds for  $x\in (h^*(t)-\delta, h^*(t))$.
		\begin{align}
			\frac{\dd}{\dd t} u(t, h(t,x))&=\hat\chi \,u(t, h(t,x))(\rho\star u)(t,h(t,x))+u(t,h(t,x))\big(1-(1+\hat\chi)u(t,h(t,x))\big)\notag \\
			&\geq u(t,h(t,x))\big(1-(1+\hat\chi)u(t,h(t,x))\big)\geq u(t, h(t, x))\left(1-\frac{1+\hat\chi}{1+\hat\chi+\alpha\chi}\right),\label{eq:propgrowth-dyn}
		\end{align}
		therefore 
		\begin{equation*}
			u(t, h(t,x))\geq u(0, x)\exp\left(\big(1-(1+\hat\chi)\eta\big)t\right),
		\end{equation*}
		provided the characteristic $h(t, x)$ does not leave the cylinder $(h^*(s)-\delta, h^*(s))$ for any $0\leq s\leq t$.

		Next by \eqref{eq:propgrowth-div} and \eqref{eq:propgrowth-usmall}, we have
		\begin{equation*}
			\frac{\dd}{\dd t}\big(h^*(t)-h(t,x)\big)\leq \chi\big(h^*(t)-h(t,x)\big)\times\eta,
		\end{equation*}
		for each $ x\in(h^*(t)-\delta, h^*(t))$. Hence  by Grönwall's Lemma 
		\begin{equation*}
			\big(h^*(t)-h(t,x)\big)\leq -xe^{\eta\chi t},
		\end{equation*}
		provided the characteristic $h(t, x)$ does not leave the cylinder $(h^*(s)-\delta, h^*(s))$ for any $0\leq s\leq t$. In particular for $0>-\frac{1}{2}\delta e^{-\eta\chi t}\geq x\geq -\delta e^{-\eta\chi t}$, we find
		\begin{align*}
			u(t, h(t,x))&\geq u(0, x)\exp\left(\left(1-\frac{1+\hat\chi}{1+\hat\chi+\alpha\chi}\right)t\right)\geq  \gamma (-x)^\alpha\exp\left(\left(1-\frac{1+\hat\chi}{1+\hat\chi+\alpha\chi}\right)t\right)\\
			&\geq\frac{1}{2^\alpha} \gamma \delta^\alpha \exp\left(\big(1-(1+\hat\chi+\alpha\chi)\eta\big) t \right)\xrightarrow[t\to+\infty]{}+\infty, 
		\end{align*}
		by our assumption that $\eta<\frac{1}{1+\hat\chi+\alpha\chi}$. This is a contradiction.

		\step We show \eqref{eq:discont2}.

		Assume by contradiction that there exists $T>0$ and $\delta>0$ such that 
		\begin{equation*}
			\sup_{t\geq T} \sup_{x\in[h^*(t)-\delta, h^*(t)]}u(t, x)< \frac{1}{1+\hat\chi+\alpha\chi}.
		\end{equation*}
		
		Since the function $u(t, x+h^*(t))$ is continuous on the compact set $[0, T]\times [-\delta, 0]$, it is uniformly continuous on this set and hence (recall that $u(t, h^*(t))=0$)  there exists $0<\delta_0\leq \delta $ such that 
		\begin{equation*}
			\sup_{t\in[0, T], x\in [-\delta_0, 0]}u(t, x+h^*(t))=\sup_{t\in[0, T], x\in [-\delta_0, 0]}\big(u(t, x+h^*(t))-u(t, h^*(t))\big)\leq\frac{1}{1+\hat\chi+\alpha\chi}.
		\end{equation*}
		Hence we  conclude
		\begin{equation*}
			\sup_{t>0, x\in [-\delta_0, 0]}u(t, x-h^*(t))\leq\frac{1}{1+\hat\chi+\alpha\chi}.
		\end{equation*}
		This is in contradiction with Step 1. Proposition \ref{prop:growth} is proved.
	\end{stepping}
\end{proof}

\begin{prop}[Refined estimate on the level sets]\label{prop:levelset}
	Let $u_0(x)$ satisfy Assumption \ref{assum:initcond-basic} and \ref{assum:initcond-steep}. Define    
	\begin{equation*}
		\xi(t, \beta):=\sup\{x\in\mathbb R\,|\, u(t, x)= \beta\}
	\end{equation*} 
	for any $ 0<\beta< \frac{1}{1+\hat\chi+\alpha\chi}  $.
	Then, the level set function $\xi(t, \beta)$ converges exponentially fast to $h^*(t)$ 
	\begin{equation}\label{eq:levelset2}
		h^*(t)-\left(\frac{\beta}{\gamma}\right)^{\frac{1}{\alpha}}e^{-\frac{\eta}{2\alpha}t}\leq \xi(t, \beta)  \leq h^*(t),
	\end{equation}
	for each $0<\beta<\frac{1}{1+\hat\chi+\alpha\chi}$, where $\eta$ is given by
	\begin{equation*}
		\eta:=1-\frac{1+\hat\chi+\alpha\chi}{\beta}\in (0,1).
	\end{equation*}
\end{prop}
\begin{proof}
	Let $\eta\in (0,1) $ be given and set $\beta^*:=\frac{1-\eta}{1+\hat\chi+\alpha\chi}$. Let us first remark that for any $ \beta \in (0,\beta^*) $, $\xi(t, \beta)$ is well-defined  by the continuity of $x\mapsto u(t,x)$ and Assumption \ref{assum:initcond-steep}, that $u(t, \xi(t, \beta))=\beta $ and that $\sup_{x\in (\xi(t, \beta), h^*(t))}u(t, x)\leq \beta$. Moreover $\xi(0, \beta)<0$ and $u_0(\xi(0,\beta))=\beta\geq \gamma \big|\xi(0, \beta)\big|^\alpha$, therefore  
	\begin{equation}\label{eq:init-lset}
		\xi(0, \beta)\geq -\left(\frac{\beta}{\gamma}\right)^{\frac{1}{\alpha}}
	\end{equation}
	for each $0<\beta\leq \beta^*=\frac{1-\eta}{1+\hat\chi+\alpha\chi}$.

	\begin{stepping}
		\step We show that if $u_0$ satisfies Assumption \ref{assum:initcond-basic} and \eqref{eq:init-lset}, then 
		\begin{equation}\label{eq:levelset-goal}
			\xi(t, \beta)\geq h^*(t)-\left(\frac{\beta}{\gamma}\right)^{\frac{1}{\alpha}}e^{\frac{\eta}{2\alpha}t},
		\end{equation}
		for all $0\leq t\leq t^*:=\frac{1}{1+\hat\chi}\ln\left(1+\frac{\eta}{2(1-\eta)}\right)$.

		Let $0<\beta\leq \beta^*$. We remark that, by Assumption \ref{assum:initcond-basic}, we have $0\leq u(t, x)\leq 1$ hence $0\leq (\rho\star u)(t, x)\leq 1$. It follows that, for all $t\geq 0$,
		\begin{equation*}
			\frac{\dd}{\dd t}u(t, h(t, x))=u(t, h(t, x))\big(1+\hat\chi\rho\star u-(1+\hat\chi)u(t, h(t, x))\big)\leq (1+\hat\chi)u(t, h(t, x)).
		\end{equation*}
		In the remaining part of Step 1 we  consider  $t\in[0, t^*]$. Using \eqref{eq:propgrowth-div} from Lemma \ref{lem:div-sep}
		we establish the following estimates on $u$ and $h$ for  $0\leq t\leq t^* $ and $ \xi\left(0, \beta^*\right)\leq x\leq 0$:

		\medskip

		$\bullet$ Since $\frac{\dd\phantom{t}}{\dd t}u(t, h(t, x))\leq (1+\hat\chi)u(t, h(t,x))$ we have $u(t, h(t, x))\leq u_0(x)e^{(1+\hat\chi)t}$ for all $t\leq t^*$  and hence if  $x\geq \xi(0,\beta^*)$,
			\begin{equation}\label{eq:lset-above}
				u(t, h(t, x))\leq \beta^*e^{\ln\left(1+\frac{\eta}{2(1-\eta)}\right)}=\frac{1-\eta}{1+\hat\chi+\alpha\chi}\left(1+\frac{\eta}{2(1-\eta)}\right)=\frac{1-\frac{\eta}{2}}{1+\hat\chi+\alpha\chi}.
			\end{equation}

		$\bullet$ Using  \eqref{eq:lset-above} in the equation along the characteristic \eqref{eq:integrated}:
		\begin{align*}
			\frac{\dd\phantom{t}}{\dd t} u(t, h(t, x))&=u(t, h(t,x))\big(1+\hat\chi(\rho\star u)(t, h(t, x))-(1+\hat\chi)u(t, h(t,x))\big)\\
			&\geq \left(1-\frac{(1+\hat\chi)(1-\frac{\eta}{2})}{1+\hat\chi+\alpha\chi}\right)u(t,h(t,x)),
		\end{align*} 
			we get 
			\begin{equation}\label{eq:lset-below}
				u(t, h(t, x))\geq u_0(x)\exp\left[ \left(1-\frac{(1+\hat\chi)(1-\frac{\eta}{2})}{1+\hat\chi+\alpha\chi}\right)t\right]
			\end{equation}

		$\bullet$ For all $x\in (\xi(0, \beta^*), 0)$, since 
		 \[ \sup_{y\in (h(t,x),h^*(t))} u(t,y) \leq \sup_{y\in (h(t,\xi(0,\beta^*)),h^*(t)} u(t,y)\leq \frac{1-\frac{\eta}{2}}{1+\hat\chi+\alpha\chi}, \]
		we have by  \eqref{eq:propgrowth-div}:
		\begin{equation*}
			h^*(t)-h(t, x)\leq \exp\left({\frac{(1-\frac{\eta}{2})\chi}{1+\hat\chi+\alpha\chi}t}\right)(h^*(0)-h(0, x)),
		\end{equation*}
		hence
			\begin{equation}\label{eq:lset-char}
				h(t, x) \geq h^*(t)+x\exp\left({\frac{(1-\frac{\eta}{2})\chi}{1+\hat\chi+\alpha\chi}t}\right).
			\end{equation}

			\medskip

		Since $\beta\leq \beta^*$, we have $\xi(0, \beta)\geq \xi(0, \beta^*)$. Using \eqref{eq:lset-below} with $x=\xi(0, \beta)$ we find that 
		\begin{equation*}
			u(t, h(t, \xi(0, \beta)))\geq \beta  \exp\left[ \left(1-\frac{(1+\hat\chi)(1-\frac{\eta}{2})}{1+\hat\chi+\alpha\chi}\right)t\right],
		\end{equation*}
		which implies 
		\begin{equation*}
			\xi\left(t, \beta  \exp\left[ \left(1-\frac{(1+\hat\chi)(1-\frac{\eta}{2})}{1+\hat\chi+\alpha\chi}\right)t\right]\right)\geq h(t, \xi(0, \beta)).
		\end{equation*}
		Now by using $x=\xi(0, \beta)$ in \eqref{eq:lset-char}, we obtain 
		\begin{equation*}
			h(t, \xi(0, \beta))\geq h^*(t)+\xi(0, \beta)\exp\left({\frac{(1-\frac{\eta}{2})\chi}{1+\hat\chi+\alpha\chi}t}\right).
		\end{equation*}
		Using \eqref{eq:init-lset}   we find that 
		\begin{equation*}
			\xi\left(0, \beta  \exp\left[- \left(1-\frac{(1+\hat\chi)(1-\frac{\eta}{2})}{1+\hat\chi+\alpha\chi}\right)t\right]\right)\geq-\left(\frac{\beta}{\gamma}\right)^{\frac{1}{\alpha}} \exp\left[-\frac{1}{\alpha} \left(1-\frac{(1+\hat\chi)(1-\frac{\eta}{2})}{1+\hat\chi+\alpha\chi}\right)t\right]
		\end{equation*}
		which leads to 
		\begin{align*}
			\xi(t, \beta)&\geq h^*(t)-\left(\frac{\beta}{\gamma}\right)^{\frac{1}{\alpha}}\exp\left[-\frac{1}{\alpha} \left(1-\frac{(1+\hat\chi)(1-\frac{\eta}{2})}{1+\hat\chi+\alpha\chi}\right)t + \frac{(1-\frac{\eta}{2})\chi}{1+\hat\chi+\alpha\chi}t\right]\\
			&=h^*(t)-\left(\frac{\beta}{\gamma}\right)^{\frac{1}{\alpha}}\exp\left[-\frac{\eta}{2\alpha}t\right]
		\end{align*}
		and this estimate holds for each $0\leq t\leq t^* $ and $0<\beta\leq \beta^*$.

		\step We show that the estimate \eqref{eq:levelset-goal} can be extended by induction.

		Define $\bar u_0(x):=u(t^*, x+h(t^*))$ and $\bar\xi(t, \beta)=\xi(t+t^*, \beta)-h^*(t^*)$. We have for each $0<\beta\leq \beta^*$
		\begin{equation*}
			\bar\xi(0, \beta)\geq- \left(\frac{\beta}{\bar\gamma}\right)^{\frac{1}{\alpha}},
		\end{equation*}
	where $\bar\gamma=\gamma e^{\frac{\eta }{2}t^*}$.  In particular the inequality \eqref{eq:init-lset} is satisfied by $\bar u_0(x)$, as well as  Assumption \ref{assum:initcond-basic}. We can apply Step 1 and \eqref{eq:levelset-goal} gives
	\begin{align*}
		\bar\xi(t, \beta)&\geq \bar h^*(t)-\left(\frac{\beta}{\bar\gamma}\right)^{\frac{1}{\alpha}}e^{-\frac{\eta}{2\alpha}t}=h(t, h^*(t))-h^*(t^*)-\left(\frac{\beta}{\gamma}\right)^{\frac{1}{\alpha}} e^{-\frac{\eta}{2\alpha}(t+t^*)}\\
		&= h^*(t+t^*) - \left(\frac{\beta}{\gamma}\right)^{\frac{1}{\alpha}} e^{-\frac{\eta}{2\alpha}(t+t^*)},
	\end{align*}
	which yields
	\begin{equation*}
		\xi(t+t^*, \beta)\geq h^*(t+t^*)- \left(\frac{\beta}{\gamma}\right)^{\frac{1}{\alpha}} e^{-\frac{\eta}{2\alpha}(t+t^*)}.
	\end{equation*}
	The proof is completed.
\end{stepping}
\end{proof}
We are now in the position to prove Theorem \ref{thm:discont}
\begin{proof}[Proof of Theorem \ref{thm:discont}]
	The first part, equation \ref{eq:discont}, has been shown in Proposition \ref{prop:growth}, while the second part (equation \eqref{eq:levelset}) has been shown in Proposition \ref{prop:levelset}.
\end{proof}

We conclude this section by the proof of Proposition \ref{prop:jump}.
\begin{proof}[Proof of Proposition \ref{prop:jump}]
	Since $x\mapsto u(t,x)$ is nonincreasing, we have $ u(t,x)\geq u(t, h^*(t))$ for each $x\leq h^*(t)$. Hence $(\rho\star u)(t, h^*(t))\geq \frac{1}{2}u(t, h^*(t))$ and 
	\begin{equation*}
		\frac{\dd}{\dd t}u(t, h^*(t))=u(t, h^*(t))\big(1+\hat\chi\,\rho\star u-(1+\hat\chi)u(t, h^*(t))\big)\geq u(t, h^*(t))\left(1-\left(1+\frac{\hat\chi}{2}\right)u(t, h^*(t))\right).
	\end{equation*}
	This yields
	\begin{equation*}
		u(t, h^*(t))\geq \frac{u_0(0)}{\left(1+\frac{\hat\chi}{2}\right)u_0(0)+e^{-t}\left(1-\left(1+\frac{\hat\chi}{2}\right)u_0(0)\right)}\underset{t\to+\infty}{\longrightarrow}\frac{1}{1+\frac{\hat\chi}{2}}=\frac{2}{2+\hat\chi}.
	\end{equation*}
	\eqref{eq:jumpsize} is shown. Next, we have $\frac{\dd \phantom{t}}{\dd t}h^*(t)=-(\rho_x\star u)(t, h^*(t))$ which gives 
	\begin{equation*}
		\frac{\dd}{\dd t}h^*(t)=\frac{\chi}{\sigma}\int_{0}^{\infty} \rho(y)u\big(t, h^*(t)-y\big)\dd y\geq u(t, h^*(t))\times \frac{\chi}{2\sigma}\underset{t\to+\infty}{\longrightarrow}\frac{\sigma\hat\chi}{2+\hat\chi}.
	\end{equation*}
	This proves \eqref{eq:charspeed-min} and  finishes the proof of Proposition \ref{prop:jump}.
\end{proof}

\section{Traveling wave solutions}
\label{sec:TW}

In this section we investigate the existence of particular solutions which consist in a fixed profile traveling at a constant speed $c$ (traveling waves). We are particularly interested in profiles which connect the stationary state $1$ near $-\infty$ to the stationary solution $0$ at a finite point of space,  say, for any $x\geq 0$.

\subsection{Existence of sharp traveling waves}

We study the traveling wave solutions of  equation \eqref{eq:main}:
\begin{equation*}
	\begin{cases}
		\partial _{t}u(t,x) -\chi\partial_x\bigl(u(t,x)\partial_x p(t,x)\bigr)=u(t,x)(1-u(t,x))\\[0.25cm]
		-\sigma^2\partial_x^2p(t,x)+p(t,x)=u(t,x)
	\end{cases}\;\;t>0,\;x\in \mathbb{R}.
\end{equation*}
Let us formally derive an equation for the traveling wave solutions to \eqref{eq:main}. We consider the traveling wave solution $ U(x-c\,t)=u(t,x) $. By using the resolvent formula of the second equation of \eqref{eq:main} formula we deduce that 
\begin{equation*}
	p(t,x)=\frac{1}{2\sigma} \int_{\R} e^{-\frac{|x-y|}{\sigma}}u(t,y)\dd y=\frac{1}{2\sigma} \int_{\R} e^{-\frac{|x-ct-l|}{\sigma}}U(l)\dd l= P(x-c\,t)
\end{equation*}
and the first equation in \eqref{eq:main} becomes
\begin{equation}
	-c \,U’(x-c\,t) -\chi\,\partial_x\Big(U(x-c\,t)\,\partial_x P(x-c\,t) \Big)=	U(x-c\,t)(1-U(x-c\,t)),\quad t>0,\;x\in \mathbb{R}.
	\label{?1.2}
\end{equation}
By developing the derivative in  \eqref{?1.2} we obtain  
 \begin{equation*}
	 \big(-c -\chi P'(x-c\,t) \big)  U'(x-c\,t) =	U(x-c\,t)(1+\hat\chi P(x-c\,t)-(1+\hat\chi)U(x-c\,t)),\quad t>0,\;x\in \mathbb{R},
 \end{equation*}
where $\hat\chi=\frac{\chi}{\sigma^2}$.
Therefore, by letting $ z=x-c\,t $, the traveling wave solutions of system \eqref{eq:main} satisfy the following equation 
\begin{equation}
	\begin{cases}
		(-c-\chi P'(z))U'(z)=U(z) \big(1+\hat\chi P(z)-(1+\hat\chi) U(z)\big),\\[0.25cm]
		-\sigma^2 P''(z)+P(z)=U(z).
	\end{cases}
	\label{?1.3}
\end{equation}
Let us finally remark that 
\begin{equation}\label{eq:P}
	P(z)=\frac{1}{2\sigma} \int_{\R} e^{-\frac{|y|}{\sigma}}U(z-y)\dd y =\frac{1}{2\sigma} \int_{\R} e^{-\frac{|z-y|}{\sigma}}U(y)\dd y.
\end{equation}
In particular if $ U $ is non-constant and  nonincreasing, then $ z \mapsto P(z) $ is strictly decreasing.

The goal of this Section is to show that equation \eqref{?1.3} can solved on the half-line $(-\infty, 0)$ which, as we will see later, will give a proof of Theorem \ref{thm:sharp-TW}. 
We begin by defining a set of admissible profiles, which is the set of function on which an appropriate fixed-point theorem will be used. The properties we impose are those who we suspect will be satisfied by the real profile of the traveling wave.
\begin{defn}\label{def:admissible}
	We say that the profile $ U:\mathbb{R} \to [0,1]$ is \textit{admissible} if
	\begin{enumerate}[(i)]
		\item $ U\in C((-\infty,0),\R)  $ and $ \lim_{z\to 0^-}U(z) $ exists and belongs to $ \left[\dfrac{2}{2+\hat\chi}, 1\right] $;
		\item $ 0\leq U(z) \leq 1 $ for any $ z\in \R $;
		\item the map $ z\mapsto U(z) $ is non-increasing on $ \R $;		
		\item $ U(z)\equiv 0 $ for any $ z\geq 0 $.
	\end{enumerate}
	We denote $\mathcal A$ the set of all admissible functions. 
\end{defn}

\begin{lem}\label{LEM:1}
	Let Assumption \ref{as:hatchi} hold and suppose that $ U $ is admissible (as in  Definition \ref{def:admissible}). Then the function $P$ defined by  $ P =(\rho \star U) $ satisfies 
	\[ P'(0)<P'(z) \leq 0, \, \text{ for all } z\in \R \backslash\{0\}. \]
	Moreover, this estimate is locally uniform in $U$ on $(-\infty, 0)$ in the sense that for each $L>1$ there is $\epsilon>0$ {\em independent of $U\in \mathcal A$} such that 
	\begin{equation*}
		P'(z)-P'(0)\geq \epsilon>0, \text{ for all } z\in\left[-L, -\frac{1}{L}\right].
	\end{equation*}
\end{lem}
\begin{proof}
	We divide the proof in five steps.\medskip\\
	\noindent\textbf{Step 1.} We prove $ P'(0)<P'(z) $ for any $ z>0 $. Notice that, for $z>0$, we have
	\[ P(z) =\frac{1}{2\sigma} \int_{-\infty}^{z} e^{-\frac{z-y}{\sigma}}U(y)\dd y + \frac{1}{2\sigma}\int_{z}^{\infty} e^{\frac{z-y}{\sigma}}U(y)\dd y=\frac{1}{2\sigma} e^{-\frac{z}{\sigma}} \int_{-\infty}^{0} e^{\frac{y}{\sigma}}U(y)\dd y. 
	\]
	Thus, taking derivative gives
$$
	P'(z)=-\frac{1}{\sigma}e^{-\frac{z}{\sigma}} \frac{1}{2\sigma} \int_{-\infty}^{0} e^{y}U(y)\dd y=e^{-\frac{z}{\sigma}}P'(0),
$$	
and since $U$ is strictly positive for negative values of $z$, we deduce that $ P'(0)<P'(z) $ for any $ z>0 $.\medskip

	\noindent\textbf{Step 2.} We prove that $ P'(0)<P'(z) $ for any $ -\sigma\ln(\frac{\hat\chi}{2})< z<0 $.	In fact, we prove the stronger result
	\[ P''(z)<0 \text{ if }  \sigma\ln\left(\frac{\hat\chi}{2}\right)< z<0. \]
	For any $ z<0 $, we have 
	\begin{align*}
		P''(z)&=\frac{1}{2\sigma^3} \int_{-\infty}^{z} e^{-\frac{z-y}{\sigma}}U(y)\dd y + \frac{1}{2\sigma^3}\int_{z}^{\infty} e^{\frac{z-y}{\sigma}}U(y)\dd y -\frac{1}{\sigma^2}U(z)\\
		&= \frac{1}{2\sigma^3} \int_{-\infty}^{z} e^{-\frac{z-y}{\sigma}}U(y)\dd y + \frac{1}{2\sigma^3}\int_{z}^{0} e^{\frac{z-y}{\sigma}}U(y)\dd y -\frac{1}{\sigma^2}U(z).
	\end{align*}
	Due to the assumption $ U\leq 1 $ and the fact that $U$ is decreasing we have
	\begin{align*}
		\sigma^2 P''(z)&\leq  \frac{1}{2\sigma } \int_{-\infty}^{z} e^{-\frac{z-y}{\sigma}}\dd y + \frac{1}{2\sigma}\int_{z}^{0} e^{\frac{z-y}{\sigma}}U(y)\dd y -U(z)\\
		&= \frac{1}{2} + \frac{1}{2\sigma}\int_{z}^{0} e^{\frac{z-y}{\sigma}}U(y)\dd y -U(z) \leq \frac{1}{2}+\frac{1}{2\sigma}\int_{z}^0e^{\frac{z-y}{\sigma}}\dd yU(z)-U(z)\\
		&= \frac{1}{2} -\frac{1}{2}\left(1+e^{\frac{z}{\sigma}}\right)U(z)\leq \frac{1}{2}\frac{2+\hat\chi-2(1+e^{\frac{z}{\sigma}})}{2+\hat\chi}=\frac{\hat\chi-2e^{\frac{z}{\sigma}}}{2(2+\hat\chi)}<0, 
	\end{align*}
	provided $z\in\left(\sigma\ln(\hat\chi/2), 0\right)$.
	In particular
	\begin{equation}\label{eq:4.5}
		P'(z)-P'(0)=-\int_{z}^0P''(y)\dd y\geq \frac{1}{\sigma(2+\hat\chi)}\left(\frac{\hat\chi}{2\sigma}z +1-e^{\frac{z}{\sigma}}\right)>0.
	\end{equation}
	
	\medskip

	\noindent\textbf{Step 3.} 
	We prove that $ P'(0)<P'(z) $ for any $  z<\sigma\ln\left( 1-\frac{\hat\chi}{2}\right) $. For any $ z< 0 $, we have
	\[ \sigma P'(z)=-\frac{1}{2\sigma} \int_{-\infty}^{z}e^{-\frac{z-y}{\sigma}}U(y)\dd y +\frac{1}{2\sigma}\int_{z}^0 e^{\frac{z-y}{\sigma}}U(y)\dd y, \quad \sigma P'(0)=-\frac{1}{2\sigma}\int_{-\infty}^{0}e^{\frac{y}{\sigma}}U(y)\dd y, \]
	and
	\[ \sigma\big(P'(z)-P'(0)\big)=\frac{1}{2\sigma}\int_{-\infty}^{0}e^{\frac{y}{\sigma}}U(y)\dd y-\frac{1}{2\sigma} \int_{-\infty}^{z}e^{-\frac{z-y}{\sigma}}U(y)\dd y +\frac{1}{2\sigma}\int_{z}^0 e^{\frac{z-y}{\sigma}}U(y)\dd y. \]
	Since for any $ z\leq0 $, $ \frac{2}{2+\hat\chi}\leq U(z) \leq 1$, we have the following estimate
	\begin{align}
		\sigma\big(P'(z)-P'(0)\big) &\geq  \frac{1}{2\sigma}\int_{-\infty}^{0}e^{\frac{y}{\sigma}}\times\frac{2}{2+\hat\chi} \dd y -\frac{1}{2\sigma} \int_{-\infty}^{z}e^{-\frac{z-y}{\sigma}}\dd y +\frac{1}{2\sigma}\int_{z}^0 e^{\frac{z-y}{\sigma}}\frac{2}{2+\hat\chi}\dd y\nonumber \\
		&= \frac{1}{2+\hat\chi} - \frac{1}{2} +\frac{1}{2+\hat\chi}\left(1-e^{\frac{z}{\sigma}}\right)\nonumber \\
		&=\frac{1}{2+\hat\chi}\left(2-e^{\frac{z}{\sigma}}-\frac{1}{2}(2+\hat\chi)\right)=\frac{1}{2+\hat\chi}\left(1-\frac{\hat\chi}{2}-e^{\frac{z}{\sigma}}\right).\label{eq:4.7}
	\end{align}
	By our assumption $ z<\sigma\ln\left( 1-\frac{\hat\chi}{2}\right) $, we deduce that $ P'(z)-P'(0)>0 $.
	
	Notice that, if $\hat\chi<1$, we have $\sigma\ln\left(\frac{\hat\chi}{2}\right)< \sigma\ln\left(1-\frac{\hat\chi}{2}\right)$ and the estimate is done. If $1\leq \hat\chi< 2$ we still need to fill a gap between the two bounds. \medskip

	\noindent\textbf{Step 4.} We assume that $\hat\chi\geq 1$ and we prove that 
	\begin{equation}\label{eq:4.6}
		P'(z)-P'(0)\geq -\int_{z}^0P''(y)\dd y\geq \frac{z}{2\sigma^2}-\frac{1}{2\sigma}\ln\left(\frac{\hat\chi}{2}\right)+\frac{1}{\sigma(2+\hat\chi)}\left(\frac{\hat\chi}{2}\ln\left(\frac{\hat\chi}{2}\right)+1-\frac{\hat\chi}{2}\right)>0 
	\end{equation}
	for any $z\in\left[\sigma\ln\left(\frac{\hat\chi}{2}\right)-\frac{\sigma}{2+\hat\chi}\left(\frac{\hat\chi}{2}\ln\left(\frac{\hat\chi}{2}\right)+1-\frac{\hat\chi}{2}\right), \sigma\ln\left(\frac{\hat\chi}{2}\right)\right]$. Notice that
	\begin{equation*}
		\frac{\hat\chi}{2}\ln\left(\frac{\hat\chi}{2}\right)+1-\frac{\hat\chi}{2}> 0,
	\end{equation*}
	because $x\mapsto x\ln(x)$ is strictly convex.

	By Step 2 we have for all $z\leq 0$:
	\begin{equation*}
		P''(z)\leq \frac{1}{2\sigma^2},
	\end{equation*}
	therefore if $z\in\left[\sigma\ln\left(\frac{\hat\chi}{2}\right)-\frac{\sigma}{2+\hat\chi}\left(\frac{\hat\chi}{2}\ln\left(\frac{\hat\chi}{2}\right)+1-\frac{\hat\chi}{2}\right), \sigma\ln\left(\frac{\hat\chi}{2}\right)\right]$ we have
	\begin{align*}
		P'(z)-P'(0)&=P'(z)-P'\left(\sigma \ln\left(\frac{\hat\chi}{2}\right)\right)+P'\left(\sigma \ln\left(\frac{\hat\chi}{2}\right)\right)-P'(0) \\
		&\geq -\int_z^{\sigma\ln\left(\frac{\hat\chi}{2}\right)}P''(y)\dd y +\frac{1}{\sigma(2+\hat\chi)}\left(\frac{\hat\chi}{2\sigma}\sigma\ln\left(\frac{\hat\chi}{2}\right)+1-\frac{\hat\chi}{2}\right) \\
		&\geq -\frac{1}{2\sigma^2}\left(\sigma\ln\left(\frac{\hat\chi}{2}\right)-z\right)+\frac{1}{\sigma(2+\hat\chi)}\left(\frac{\hat\chi}{2}\ln\left(\frac{\hat\chi}{2}\right)+1-\frac{\hat\chi}{2}\right)\\ 
		&\geq \frac{z}{2\sigma^2}-\frac{\ln\left(\frac{\hat\chi}{2}\right)}{2\sigma}+\frac{1}{\sigma(2+\hat\chi)}\left(\frac{\hat\chi}{2}\ln\left(\frac{\hat\chi}{2}\right)+1-\frac{\hat\chi}{2}\right)>0.
	\end{align*}
	We have proved the desired estimate.

	\medskip

	\noindent\textbf{Step 5.} We show the local uniformity. If $\hat\chi<1$ the local uniformity follows from Step 2 and Step 3 because $1-\frac{\hat\chi}{2}<\frac{\hat\chi}{2}$. If $1\leq\hat\chi<2$, then 
	\begin{equation}\label{eq:asonhatchi}
		\ln\left(\frac{\hat\chi}{2}\right)-\frac{2}{2+\hat\chi}\left(\frac{\hat\chi}{2}\ln\left(\frac{\hat\chi}{2}\right) +1-\frac{\hat\chi}{2}\right)<\ln\left(1-\frac{\hat\chi}{2}\right),
	\end{equation}
	because of  Assumption \ref{as:hatchi} and Lemma \ref{lem:B1} (notice that \eqref{eq:asonhatchi} is equivalent to $f(\hat\chi)<0$, where $f$ is as defined in Lemma \ref{lem:B1}). 
	By the estimates \eqref{eq:4.5}, \eqref{eq:4.7} and  \eqref{eq:4.6} from Step 2, Step 3 and Step 4, we find that $P'(z)-P'(0)>0$ on every compact subset of  $(-\infty, 0)$ and is bounded from below by a constant independent of $U$. This finishes the proof of Lemma \ref{LEM:1}.
\end{proof}
Before resuming to the proof, let us define the mapping $\mathcal T$ to which we want to apply a fixed-point theorem. Fix $U\in\mathcal A$, we define $\mathcal T(U)$ as 
\begin{equation}\label{eq:defT1}
	\mathcal T(U)(z):=\mathcal U(\tau^{-1}(z)) \text{  for all } z<0
\end{equation}
and $\mathcal T(U)(z)\equiv 0$ for all $z\geq 0$, where $ \tau:\mathbb{R} \mapsto (-\infty,0) $ is the solution of the following scalar ordinary differential equation
\begin{equation}\label{eq:tau}
\begin{cases}
	\tau'(t) &= \chi\big(P'(0) -P'(\tau (t))\big),\\
\tau(0) &= -1,
\end{cases}
\end{equation} 
and 
\[ 
\mathcal{U}(t) = \bigg[(1+\hat\chi)\int_{-\infty}^{t} \exp\left({-\int_{l}^{t}1+\hat\chi P(\tau(s))\dd s}\right) \dd l \bigg]^{-1}, \forall t \in \mathbb{R}.
\]

\begin{lem}[Stability of $\mathcal A$]\label{LEM:2}
	Let Assumption \ref{as:hatchi} be satisfied, let $ U $ be admissible in the sense of Definition \ref{def:admissible} and $\mathcal T$ be the map defined by \eqref{eq:defT1}.
	Then the image of $U$ by $\mathcal T$ has the following properties: 
	\begin{enumerate}[label={\rm(\roman*)}]
		\item\label{item:LEM2i}  $ \dfrac{2}{2+\hat\chi}\leq \mathcal{T} (U)(z)\leq 1$ for all $z\leq 0 $;
		\item\label{item:LEM2ii} $ \mathcal{T} (U) $ is strictly decreasing on $ (-\infty,0] $;
		\item\label{item:LEM2iii}$ \mathcal{T} (U)\in C^1((-\infty,0),\R) $ and $\mathcal T(U)(0^-)= \lim_{z\to 0^-} \mathcal{T} (U)(z) =\dfrac{1+\hat\chi P(0)}{1+\hat\chi} $.
	\end{enumerate}
In particular, $\mathcal A $ is left stable by $\mathcal T$ 
$$
\mathcal T\left( \mathcal A \right)  \subset \mathcal A.
$$ 
\end{lem}
\begin{proof}
	We divide the proof in three steps.\medskip\\
	\noindent\textbf{Step 1.} We prove that $ \dfrac{2}{2+\hat\chi}\leq \mathcal{T} (U)(z)\leq 1 $ for all $z<0$. For any $z\in\mathbb R$ we have 
	\begin{align*}
		P(z) &= \int_{-\infty}^{\infty} \rho(y)U(z-y) \dd y \leq \int_{-\infty}^{+\infty}\rho(y)\dd y=1, \\
		P(z) &= \int_{-\infty}^{\infty} \rho(y)U(z-y) \dd y\geq 0.
	\end{align*}
	Since $ \dfrac{2}{2+\hat\chi}\leq U(z) \leq 1$ for all $z<0$, we have for $z<0$ 
	\[ P(z) \geq  \frac{1}{2\sigma}\int_{z}^{+\infty} \exp\left(-\frac{|y|}{\sigma}\right) \times\frac{2}{2+\hat\chi} \dd y =\frac{2}{2+\hat\chi}\left(1-\dfrac{e^{\frac{z}{\sigma}}}{2}\right) \geq\frac{1}{2+\hat\chi}.  \]
	Thus, for any $ z\leq0 $, we have  $\dfrac{1}{2+\hat\chi}\leq P(z)\leq 1 $.
	Since $ \tau(t)  $ is the solution of 
	\begin{equation*}
	\begin{cases}
		\tau'(t) &= \chi\big(P'(0) -P'(\tau (t))\big)\\
		\tau(0) &= -1, 
	\end{cases}
	\end{equation*}
	and due to Lemma \ref{LEM:1}, $ t\to \tau (t) $ is  strictly decreasing, continuous and 
	\[ \lim_{t\to-\infty}\tau (t) = 0,\quad \lim_{t\to+\infty} \tau (t) =-\infty. \]
	Therefore, 
	\[ \frac{1}{2+\hat\chi}\leq P(\tau(t))\leq 1, \quad t\in \R.\]
	Since by definition  $ \mathcal{U}(t) =  \bigg[(1+\hat\chi)\int_{-\infty}^{t} e^{-\int_{l}^{t}1+\hat\chi P(\tau(s))\dd s} \dd l \bigg]^{-1} $, $ \mathcal{U} $ is monotone with respect to $ P$, and we compute  on the one hand
	\begin{align*}
		\mathcal{U}(t)&\leq \bigg[(1+\hat\chi)\int_{-\infty}^{t} e^{-\int_{l}^{t}1+\hat\chi\dd s} \dd l \bigg]^{-1}\\
		&=\bigg[(1+\hat\chi)\int_{-\infty}^{t} e^{-(1+\hat\chi)(t-l)} \dd l \bigg]^{-1} =1.
	\end{align*}
	On the other hand, we can see that
	\begin{align*}
		\mathcal{U}(t)&\geq \bigg[(1+\hat\chi)\int_{-\infty}^{t} \exp\left(-\int_{l}^{t}1+\frac{\hat\chi}{2+\hat\chi}\dd s\right) \dd l \bigg]^{-1}\\
		&=\bigg[(1+\hat\chi)\int_{-\infty}^{t} \exp\left(-\left(1+\frac{\hat\chi}{2+\hat\chi}\right)(t-l)\right) \dd l \bigg]^{-1}=\frac{2}{2+\hat\chi}.
	\end{align*}
	This implies $ \frac{2}{2+\hat\chi}\leq \mathcal{U}(t)  \leq 1,\,\forall t\in \R $. Since  $ \tau^{-1}$ maps $(-\infty,0) $ to $ \R $, for any $ z<0 $ we have indeed
	\[ \frac{2}{2+\hat\chi}\leq\mathcal{T}(U)(z) = \mathcal{U}(\tau^{-1}(z)) \leq 1.\]
	Item \ref{item:LEM2i} is proved.\medskip

	\noindent\textbf{Step 2.} We prove that $ z \mapsto \mathcal{T}(U)(z) $ is strictly decreasing on $ (-\infty,0) $. First, we prove that  $ t\mapsto \mathcal{U}(t) $ is strictly increasing. Indeed $\mathcal U$ is differentiable and we have 
	\begin{align}\label{eq:calUprime}
		\mathcal{U}'(t) &=  \frac{-1}{1+\hat\chi}\times\dfrac{\displaystyle 1+\int_{-\infty}^{t} -\big(1+\hat\chi P(\tau(t))\big) e^{-\int_{l}^{t}1+\hat\chi P(\tau(s))\dd s} \dd l }{\bigg[\displaystyle\int_{-\infty}^{t} \exp\left(-\int_{l}^{t}1+P(\tau(s))\dd s\right) \dd l \bigg]^{2}}
	\end{align}
	Moreover, for any $ l<t $,  we have $ \tau(t) <\tau(l) $. Since $ P $ is strictly decreasing, $ P(\tau(l))< P(\tau(t))  $. We deduce 
	\begin{align*}
	\int_{-\infty}^{t} e^{-\int_{l}^{t}1+\hat\chi P(\tau(s))\dd s}\big(1+\hat\chi P(\tau(t))\big) \dd l 
	&> \int_{-\infty}^{t} e^{-\int_{l}^{t}1+\hat\chi P(\tau(s))\dd s}\big(1+\hat\chi P(\tau(l))\big) \dd l \\
	&= \int_{-\infty}^{t} \frac{\dd}{\dd l} \Big( e^{-\int_{l}^{t}1+\hat\chi P(\tau(s))\dd s}\Big) =1.  
	\end{align*}
	This implies $ \mathcal{U}'(t)>0 $ and $ t\mapsto \mathcal{U}(t) $ is strictly increasing.  Note that the inverse map $ z\mapsto \tau^{-1}(z) $ is strictly decreasing, therefore the composition of two mappings
	\[ z\longmapsto\mathcal{T}(U)(z) = \mathcal{U}(\tau^{-1}(z))  \] 
	is also strictly decreasing on $ (-\infty,0) $. Item \ref{item:LEM2ii} is proved. 
	\medskip

	\noindent\textbf{Step 3.} We prove that $ \mathcal{T} (U)\in C^1((-\infty,0),\R) $ and compute the limit of $\mathcal T(U)$ as $z\to 0^-$.

	Since for any $ z<0 $  
	\[ \sigma^2P''(z) = -U(z)+P(z) \in C((-\infty,0),\R), \]
	$ P $ belongs to $ C^2((-\infty,0),\R) $, which implies that $ t\mapsto \tau(t)  $ belongs to $ C^1(\R,(-\infty,0)) $.
	By \eqref{eq:calUprime},  the function $ t\mapsto \mathcal{U}'(t) $
	is continuous and the inverse map $ z\to \tau^{-1}(z) $ is also of class $ C^1 $ from $ (-\infty,0) $ to $ \R $. Thus, the function 
	\[ z\longmapsto \mathcal{T}(U)(z) = \mathcal{U}(\tau^{-1}(z)) \]
	is of class $ C^1 $ from $ (-\infty,0) $ to $ \R $. Moreover, the map $t\mapsto \mathcal{U}(t) $ is strictly decreasing and is bounded from below by $ \frac{2}{2+\hat\chi}>0 $, thus $ \lim_{t\to -\infty}\mathcal{U}(t) $ exists. In particular 
	\[ \mathcal{T}(U)(0^-) :=\lim_{z\to 0^-} \mathcal{U}(\tau^{-1}(z)) =\lim_{t\to-\infty} \mathcal{U}(t). \]
	{By the definition of $ \mathcal{U} $
	\begin{align*}
	\mathcal{T}(U)(0^-) &=\lim_{t\to -\infty} \mathcal{U}(t)\\
		&= \lim_{t\to - \infty} \bigg[(1+\hat\chi)\int_{-\infty}^{t} e^{-\int_{l}^{t}1+\hat\chi P(\tau(s))\dd s} \dd l \bigg]^{-1}\\
		&=\lim_{t\to - \infty}\frac{e^{\int_{0}^{t}1+\hat\chi P(\tau(s))\dd s}}{(1+\hat\chi)\int_{-\infty}^{t} e^{\int_{0}^{l}1+\hat\chi P(\tau(s))\dd s} \dd l}.
	\end{align*}
	By employing L'H\^opital rule
	\begin{align*}
		\mathcal{T}(U)(0^-)&=\lim_{t\to - \infty}\frac{e^{\int_{0}^{t}1+\hat\chi P(\tau(s))\dd s}}{(1+\hat\chi)\int_{-\infty}^{t} e^{\int_{0}^{l}1+\hat\chi P(\tau(s))\dd s} \dd l}\\
		&= \lim_{t\to - \infty} \frac{\big(1+\hat\chi P(\tau(t))\big) e^{\int_{0}^{t}1+P(\tau(s))\dd s}}{(1+\hat\chi) \,e^{\int_{0}^{t}1+P(\tau(s))\dd s}}\\
	&=\frac{1+\hat\chi P(0)}{1+\hat\chi}.
	\end{align*}
	}
	Therefore, $ \mathcal{T} (U)\in C^1((-\infty,0),\R) \cap  C((-\infty,0],\R) $ and $ \mathcal{T}(U)(0)=(1+\hat\chi P(0))/(1+\hat\chi) $. This proves Item \ref{item:LEM2iii} and concludes the proof of Lemma \ref{LEM:2}.
\end{proof}
Next we focus on the continuity of $\mathcal T$ for a particular topology. 
\begin{lem}[Continuity of $\mathcal T$]\label{lem:continuity}
Define the weighted norm 
	\begin{equation}\label{eq:norm-eta}
	\Vert U\Vert_\eta:=\sup_{z \in (-\infty,0)} \alpha(z)|U(z)|,
	\end{equation}
	where
	$$
	\alpha(z):=\sqrt{-z}\,e^{\eta z}\leq \frac{1}{\sqrt{2e\eta}}, \text{ for all } z \leq 0,
	$$
	with $0<\eta<\sigma^{-1}$.
	If Assumption \ref{as:hatchi} is satisfied, then the map $\mathcal T$ is continuous on $\mathcal A$ for the distance induced by $\Vert\cdot\Vert_\eta$.
\end{lem}
\begin{proof}
	Let $U\in \mathcal A$ and $\varepsilon>0$ be given. Let $\tilde U\in \mathcal A$ be given and define the corresponding pressure and rescaled variable $\tilde P:=\rho\star \tilde U$ and $\tilde\tau$ as the solution to \eqref{eq:tau} with $U$ replaced by $\tilde U$. We remark that :
	\begin{multline*}
		|\mathcal T(U)(z)-\mathcal T(\tilde U)(z)|=\\
		|\mathcal T(U)(z)\mathcal T(\tilde U)(z)|\left| \int_{-\infty}^{\tilde\tau^{-1}(z)}e^{-\int_l^{\tilde \tau^{-1}(z)}1+\hat\chi \tilde P(\tilde \tau(s)\dd s}\dd l-\int_{-\infty}^{ \tau^{-1}(z)}e^{-\int_l^{\tau^{-1}(z)}1+\hat\chi P(\tau(s)\dd s}\dd l\right|\\
		\leq \left| \int_{-\infty}^{\tilde\tau^{-1}(z)}e^{-\int_l^{\tilde \tau^{-1}(z)}1+\hat\chi \tilde P(\tilde \tau(s))\dd s}\dd l-\int_{-\infty}^{ \tau^{-1}(z)}e^{-\int_l^{\tau^{-1}(z)}1+\hat\chi P(\tau(s))\dd s}\dd l\right|,
	\end{multline*}
	by Lemma \ref{LEM:2}. Define $T_{-L}(U):=\int_{-L}^{ \tau^{-1}(z)}e^{-\int_l^{\tau^{-1}(z)}1+\hat\chi P(\tau(s))\dd s}\dd l$. We have $\mathcal T(U)=T_{-\infty}(U)$ and 
	\begin{align*}
		|T_{-\infty}(U)-T_{-\infty}(\tilde U)|&\leq \left| \int_{-\infty}^{\tilde\tau^{-1}(z)-L}e^{-\int_l^{\tilde \tau^{-1}(z)}1+\hat\chi \tilde P(\tilde \tau(s))\dd s}\dd l-\int_{-\infty}^{ \tau^{-1}(z)-L}e^{-\int_l^{\tau^{-1}(z)}1+\hat\chi P(\tau(s))\dd s}\dd l\right| \\
		&\quad +\left| \int_{\tilde\tau^{-1}(z)-L}^{\tilde\tau^{-1}(z)}e^{-\int_l^{\tilde \tau^{-1}(z)}1+\hat\chi \tilde P(\tilde \tau(s))\dd s}\dd l-\int_{\tau^{-1}(z)-L}^{ \tau^{-1}(z)}e^{-\int_l^{\tau^{-1}(z)}1+\hat\chi P(\tau(s))\dd s}\dd l\right| \\
		&\leq e^{-L}+e^{-L}+\left| \int_{\tilde\tau^{-1}(z)-L}^{\tilde\tau^{-1}(z)}e^{-\int_l^{\tilde \tau^{-1}(z)}1+\tilde \hat\chi P(\tilde \tau(s))\dd s}\dd l-\int_{\tau^{-1}(z)-L}^{ \tau^{-1}(z)}e^{-\int_l^{\tau^{-1}(z)}1+\hat\chi P(\tau(s))\dd s}\dd l\right|\\
		&\leq \frac{\varepsilon}{2}\sqrt{2\eta e}+\left| \int_{\tilde\tau^{-1}(z)-L}^{\tilde\tau^{-1}(z)}e^{-\int_l^{\tilde \tau^{-1}(z)}1+\hat\chi \tilde P(\tilde \tau(s))\dd s}\dd l-\int_{\tau^{-1}(z)-L}^{ \tau^{-1}(z)}e^{-\int_l^{\tau^{-1}(z)}1+\hat\chi P(\tau(s))\dd s}\dd l\right|\\
		&=\frac{\varepsilon}{2}\sqrt{2\eta e}+|T_{-L}(U)(z)-T_{-L}(\tilde U)(z)|,
	\end{align*}
	for $L:=-\ln\left(\frac{\varepsilon}{2}\sqrt{\frac{e\eta}{2}}\right)>0$. \medskip

	Let $z_0$ and $z_1$ be respectively the smallest and the biggest negative root of the equation 
	\begin{equation*}
		\eta z+\frac{1}{2}\ln(-z)=\ln\left(\frac{\varepsilon}{4}\right).
	\end{equation*}
	Then if $z\not\in[ z_0, z_1]$ we have $\sqrt{-z}e^{\eta z}\leq \frac{\varepsilon}{4}$ and, since $|T_{-L}(U)|\leq 1$  we have  
	\begin{align*}
		\sqrt{-z}e^{\eta z}|T_{-L}(U)(z)|&=\sqrt{-z}e^{\eta z}\left|\int_{\tau^{-1}(z)-L}^{ \tau^{-1}(z)}e^{-\int_l^{\tau^{-1}(z)}1+\hat\chi P(\tau(s))\dd s}\dd l\right|\\
		&\leq \frac{\varepsilon}{4}\int_{\tau^{-1}(z)-L}^{ \tau^{-1}(z)}e^{-\int_l^{\tau^{-1}(z)}1\dd s} \dd l=\frac{\varepsilon}{4} (1-e^{-L}) \leq \frac{\varepsilon}{4}.
	\end{align*}
	Similarly, we have
	\begin{equation*}	
		\sqrt{-z}e^{\eta z}|T_{-L}(\tilde U)(z)|\leq \frac{\varepsilon}{4}.
	\end{equation*}
	We have shown 
	\begin{equation*}
		\sup_{z\not\in[z_0, z_1]}\sqrt{-z}e^{\eta z}|\mathcal T(U)(z)-\mathcal T(\tilde U)(z)|\leq \varepsilon.
	\end{equation*}\medskip

	There remains to estimate $\sqrt{-z}e^{\eta z}|T_{-L}(U)(z)-T_{-L}(\tilde U)(z)|$ when $z\in[z_0, z_1]$. We have
	\begin{align*}
		|T_{-L}(U)(z)-T_{-L}(\tilde U)(z)|&=\left| \int_{\tilde\tau^{-1}(z)-L}^{\tilde\tau^{-1}(z)}e^{-\int_{l}^{\tilde \tau^{-1}(z)} 1+\hat\chi \tilde P(\tilde \tau(s))\dd s}   \dd l-\int_{\tau^{-1}(z)-L}^{ \tau^{-1}(z)}e^{-\int_{l}^{\tau^{-1}(z)} 1+\hat\chi P(\tau(s))\dd s}\dd l\right|\\
		&\leq 2|\tilde \tau^{-1}(z)-\tau^{-1}(z)|\\
		&\quad +\left|\int_{\tau^{-1}(z)-L}^{\tau^{-1}(z)}e^{-\int_{l}^{\tilde \tau^{-1}(z)} 1+\hat\chi \tilde P(\tilde \tau(s))\dd s}-e^{-\int_{l}^{\tau^{-1}(z)} 1+\hat\chi P(\tau(s))\dd s}\dd l\right|\\
		&\leq 2|\tilde \tau^{-1}(z)-\tau^{-1}(z)| \\ 
		&\quad + L\sup_{l\in(\tau^{-1}(z)-L, \tau^{-1}(z))}\left|e^{ \int_{l}^{\tau^{-1}(z)}1+\hat\chi P(\tau(s))\dd s-\int_{l}^{\tilde \tau^{-1}(z)}1+\hat\chi \tilde P(\tilde \tau(s))\dd s}-1\right|, 
	\end{align*}
	and we remark that 
	\begin{multline*}
		\left|\int_l^{\tau^{-1}(z)}1+\hat\chi P(\tau(s))\dd s-\int_l^{\tilde \tau^{-1}(z)}1+\hat\chi \tilde P(\tilde \tau(s))\dd s\right| \\
	\begin{aligned}
		\relax&\leq 2|\tau^{-1}(z)-\tilde \tau^{-1}(z)|+\hat\chi\left|\int_l^{\tau^{-1}(z)}P(\tau(s))-\tilde P(\tilde \tau(s))\dd s\right|\\
		&\leq 2|\tau^{-1}(z)-\tilde \tau^{-1}(z)|+\hat\chi L\sup_{s\in(\tau^{-1}(z)-L, \tau^{-1}(z))}|P(\tau(s))-P(\tilde \tau(s))| \\
		&\quad +\hat\chi L\sup_{s\in(\tau^{-1}(z)-L, \tau^{-1}(z))}|P(\tilde \tau(s))-\tilde P(\tilde \tau(s))| .
	\end{aligned}
	\end{multline*}
	To conclude the proof of the continuity of $\mathcal T$, we show that each of those three terms can be made arbitrarily small (uniformly on $[z_0, z_1]$) by choosing $\tilde U$ sufficiently close to $U$ in the $\Vert \cdot\Vert_{\eta}$ norm. 
	We start with  the second one. We have for all $z\leq 0$:
	\begin{align*}
		|P(z)-\tilde P(z)|&=\frac{1}{2\sigma}\left|\int_{-\infty}^0e^{-\frac{|z-y|}{\sigma}}(U(y)-\tilde U(y))\dd y\right|\\
		&\leq \frac{1}{2\sigma}\int_{-\infty}^ze^{\frac{y-z}{\sigma}}|U(y)-\tilde U(y)|\dd y+\frac{1}{2\sigma}\int_z^0e^{\frac{z-y}{\sigma}}|U(y)-\tilde U(y)|\dd y\\
		&\leq \frac{1}{2\sigma}\sqrt{\frac{2\eta}{e}}e^{-\frac{z}{\sigma}}\int_{-\infty}^z\frac{e^{(1-\sigma\eta)\frac{y}{\sigma}}}{\sqrt{-y}}\Vert U-\tilde U\Vert_{\eta}\dd y+\frac{1}{2}\sqrt{\frac{2\eta}{e}}e^{\frac{z}{\sigma}}\int_z^0 \frac{e^{-(1+\sigma\eta)\frac{y}{\sigma}}}{\sqrt{-y}}\Vert U-\tilde U\Vert_\eta\dd y\\
		&=\sigma^{-1}\sqrt{\frac{\eta}{2e}}\left[e^{-\frac{z}{\sigma}}\int_{-\infty}^{z}\frac{e^{(1-\sigma\eta)\frac{y}{\sigma}}}{\sqrt{-y}}\dd y+e^{\frac{z}{\sigma}}\int_{z}^0\frac{e^{-(1+\sigma\eta)\frac{y}{\sigma}}}{\sqrt{-y}}\dd y \right]\Vert U-\tilde U\Vert_\eta \\
		&=:C_P(z)\Vert U-\tilde U\Vert_\eta.
	\end{align*}
	A similar computation shows that, for all $z\leq 0$,
	\begin{equation*}
		|P'(z)-\tilde P'(z)|\leq\sigma^{-2}\sqrt{\frac{\eta}{2e}}\left[e^{-\frac{z}{\sigma}}\int_{-\infty}^{z}\frac{e^{(1-\sigma\eta)\frac{y}{\sigma}}}{\sqrt{-y}}\dd y+e^{\frac{z}{\sigma}}\int_{z}^0\frac{e^{-(1+\sigma\eta)\frac{y}{\sigma}}}{\sqrt{-y}}\dd y \right]\Vert U-\tilde U\Vert_\eta=\frac{1}{\sigma}C_P(z)\Vert U-\tilde U\Vert_\eta. 
	\end{equation*}
	In particular for $z=0$ we have 
	\begin{equation*}
		|P'(0)-\tilde P'(0)|\leq \sigma^{-2}\sqrt{\frac{\eta}{2e}}\int_{-\infty}^0\frac{e^{(1-\sigma\eta)\frac{y}{\sigma}}}{\sqrt{-y}}\dd y\Vert U-\tilde U\Vert_{\eta},
	\end{equation*}
	and therefore $P'(0)$ and $\tilde P'(0)$ can be chosen arbitrarily small. Next we show that $\tau(t)$ and $\tilde \tau(t)$ are uniformly close for $t\in[\tau^{-1}(z_0)-L, \tau^{-1}(z_1)]$. Indeed, we compute: 
	\begin{align*}
		|(\tau-\tilde \tau)(t)|&=\chi\left|\int_{0}^tP'(0)-P'(\tau(s))\dd s - \int_{0}^t\tilde P'(0)-\tilde P'(\tilde \tau(s))\dd s\right|\\
		&\leq \chi\left|t(P'(0)-\tilde P'(0))+\int_0^t\tilde P'(\tau(s))-P'(\tau(s))\dd s\right|+\chi\left|\int_0^t\tilde P'(\tilde \tau(s))-\tilde P'(\tau(s))\dd s\right|\\ 
		&\leq \chi t[C_P(0)+\max_{0\leq s\leq t}C_P(\tau(s))]\Vert U-\tilde U\Vert_\eta+\hat \chi\int_0^t|\tilde \tau(s)-\tau(s)|\dd s,
	\end{align*}
	where we have used the fact that $ \sigma^2|P''(z)|=|P(z)-U(z)|\leq 1$. By Grönwall's Lemma, we have therefore
	\begin{equation*}
		|\tau(t)-\tilde \tau(t)|\leq \chi t\big[C_P(0)+\max_{0\leq s\leq t}C_P(\tau(s))\big]\Vert U-\tilde U\Vert_\eta e^{\hat\chi t},
	\end{equation*}
	and we have shown that $\tau$ and $\tilde \tau $ can be made arbitrarily close  by choosing $\Vert U-\tilde U\Vert_\eta$ sufficiently small. This gives an arbitrary control on the term  
	\begin{equation*}
		\sup_{s\in(\tau^{-1}(z)-L, \tau^{-1}(z))}|P(\tau(s))-P(\tilde \tau(s))|\leq |P'(0)| |\tau(s)-\tilde \tau(s)|,
	\end{equation*}
	since $P'(0)< P'(z)\leq 0 $ by Lemma \ref{LEM:1}, and on the term 
	\begin{equation*}
		\sup_{s\in(\tau^{-1}(z)-L, \tau^{-1}(z))}|P(\tilde \tau(s))-\tilde P(\tilde \tau(s))|\leq  \left[\sup_{s\in(\tau^{-1}(z)-L, \tau^{-1}(z))}C_P(\tilde \tau(s))\right]\Vert U-\tilde U\Vert_\eta .
	\end{equation*}

	Finally, we estimate $\tau^{-1}(z)-\tilde \tau^{-1}(z)$ by the remark:
	\begin{align*}
		|\tau^{-1}(z)-\tilde \tau^{-1}(z)|&=\left|\int_{-1}^z\frac{1}{\tau'(\tau^{-1}(y))}\dd y-\int_{-1}^z\frac{1}{\tilde\tau'(\tilde\tau^{-1}(y))}\dd y \right| \\
		&= \frac{1}{\chi}\left|\int_{-1}^z\frac{1}{P'(0)-P'(y)}-\frac{1}{\tilde P'(0)-\tilde P'(y)}\dd y \right| \\
		&\leq \frac{1}{\chi}\int_{-1}^z\frac{|P'(0)-\tilde P'(0)|+|P'(y)-\tilde P'(y)|}{|P'(0)-P'(y)||\tilde P'(0)-\tilde P'(y)|}\dd y,
	\end{align*}
	recalling that we have a uniform lower bound for $|P'(0)-P'(y)|$ and $|\tilde P'(0)-\tilde P'(y)|$ by Lemma \ref{LEM:1}.

This finishes the proof of Lemma \ref{lem:continuity}.
\end{proof}
 
\begin{lem}\label{LEM:3}
	Suppose $ U $ is admissible in the sense of Definition \ref{def:admissible} and that Assumption \ref{as:hatchi} holds.
	Then $ \mathcal{T}(U)\in C^1((-\infty,0],\R)$ and 
	\begin{equation}\label{eq:Vprime}
		\mathcal{T} (U)'(z) =\mathcal{T} (U)(z)\frac{1+\hat\chi P(z)-(1+\hat\chi)\mathcal{T} (U)(z)}{\chi\big(P'(0)-P'(z)\big)},\quad \forall z< 0.
	\end{equation}
	Moreover  
	\[ \lim_{z\to 0^-}\mathcal{T} (U)'(z)=\frac{P'(0)}{1+\hat\chi}\frac{1+\hat\chi P(0)}{1+\hat\chi U(0^-)}. \]
\end{lem}
\begin{proof}
	We divide the proof in two steps.\medskip

	\noindent\textbf{Step 1.} We prove \eqref{eq:Vprime}. 
	
We observe that 
$$ 
	\tau'(\tau^{-1}(z)):=\chi\big(P'(0)-P'(z)\big), 
$$
	therefore $\mathcal T(U)$ is differentiable for each $z<0$ and  
	\[ \mathcal{T} (U)'(z) = \mathcal{U}'(\tau^{-1}(z))\frac{1}{\tau'(\tau^{-1}(z))}=\mathcal{U}'(\tau^{-1}(z))\frac{1}{\chi\big(P'(0)-P'(z)\big)}. \]
	By Equation \eqref{eq:calUprime} in Lemma \ref{LEM:2} we have 
	\begin{align*}
	\mathcal{U}'(t) &= \frac{1}{1+\hat\chi}\bigg[\int_{-\infty}^{t} e^{-\int_{l}^{t}1+\hat\chi P(\tau(s))\dd s} \dd l \bigg]^{-2} \\
	&\quad \times \Big(\int_{-\infty}^{t} e^{-\int_{l}^{t}1+\hat\chi P(\tau(s))\dd s}\big(1+\hat\chi P(\tau(t))\big) \dd l -1\Big)\\
		&=  \bigg[(1+\hat\chi)\int_{-\infty}^{t} e^{-\int_{l}^{t}1+\hat\chi P(\tau(s))\dd s} \dd l \bigg]^{-2} \\
		&\quad \times \Big( (1+\hat\chi) \int_{-\infty}^{t} e^{-\int_{l}^{t}1+\hat\chi P(\tau(s))\dd s}\dd l \big(1+\hat\chi P(\tau(t))\big)  - (1+\hat\chi)\Big)\\
		&=\mathcal{U}^2(t) \Big(\,\mathcal{U}^{-1}(t)\big(1+\hat\chi P(\tau(t))\big) - (1+\hat\chi)\Big)\\
		&=\mathcal{U}(t)\Big(1+\hat\chi P(\tau(t)) -(1+\hat\chi)\,\mathcal{U}(t)\Big).
	\end{align*}
	Therefore, we can rewrite $ \mathcal{T} (U)'(z) $ as 
	\begin{align*}
		\mathcal{T} (U)'(z)&=\frac{\mathcal{U}'(\tau^{-1}(z))}{\chi\big(P'(0)-P'(z)\big)}\\
		&= \mathcal{U}(\tau^{-1}(z))\frac{1+\hat\chi P(z) -(1+\hat\chi)\,\mathcal{U}(\tau^{-1}(z))}{\chi\big(P'(0)-P'(z)\big)}\\
		&=\mathcal{T} (U)(z)\frac{1+\hat\chi P(z)-(1+\hat\chi)\mathcal{T} (U)(z)}{\chi\big(P'(0)-P'(z)\big)}.
	\end{align*}
	Equation \eqref{eq:Vprime} follows. \medskip

	\noindent\textbf{Step 2.} Next we prove
	\[ \lim_{z\to 0^-}\mathcal{T} (U)'(z)=\frac{P'(0)}{1+\hat\chi }\frac{1+\hat\chi P(0)}{1+\hat\chi U(0)}. \]
	Recall that 
	\begin{equation*}
		\mathcal T(U)(z)=\mathcal U(\tau^{-1}(z))=\dfrac{1}{(1+\hat\chi)\int_{-\infty}^{\tau^{-1}(z)} e^{-\int_{l}^{\tau^{-1}(z)}1+\hat\chi P(\tau(s))\dd s} \dd l}=\dfrac{e^{\int_0^{\tau^{-1}(z)}1+\hat\chi P(\tau(s))\dd s}}{(1+\hat\chi)\int_{-\infty}^{\tau^{-1}(z)}e^{\int_0^{l}1+\hat\chi P(\tau(s))\dd s}\dd l}.
	\end{equation*}
	We have shown in Step 1 that for any $ z<0 $
	\begin{align}\label{eq:LEM3-0}
		\mathcal{T} (U)'(z)=\mathcal{T} (U)(z)\frac{1+\hat\chi P(z)-(1+\hat\chi)\mathcal{T} (U)(z)}{\chi\big(P'(0)-P'(z)\big)},
	\end{align}
	and by Lemma \ref{LEM:2} we have 
	\[ \lim_{z\to 0^-}\mathcal{T} (U)(z)=\frac{1+\hat\chi P(0)}{1+\hat\chi}. \]
	Moreover, 
	\begin{align*}
		\dfrac{1+\hat\chi P(z)-(1+\hat\chi)\mathcal T_1(U)(z)}{\chi \big(P'(0)-P'(z)\big)}&=\dfrac{(1+\hat\chi P(z))\int_{-\infty}^{\tau^{-1}(z)}e^{\int_0^l 1+\hat\chi P(\tau(s))\dd s}\dd l-e^{\int_0^{\tau^{-1}(z)}1+\hat\chi P(\tau(s))\dd s}}{\chi(P'(0)-P'(z))\int_{-\infty}^{\tau^{-1}(z)}e^{\int_0^l1+\hat\chi P(\tau(s))\dd s}}=:\frac{N(z)}{D(z)},
	\end{align*}
	and
	\begin{align*}
		\dfrac{N'(z)}{D'(z)}&=\frac{\hat\chi P'(z)\int_{-\infty}^{\tau^{-1}(z)}e^{\int_0^l1+\hat\chi P(\tau(s))\dd s}\dd l} {-\chi P''(z)\int_{-\infty}^{\tau^{-1}(z)}e^{\int_0^l1+\hat\chi P(\tau(s))\dd s}\dd s + \chi\big(P'(0)-P'(z)\big)(\tau^{-1})'(z)e^{\int_0^{\tau^{-1}(z)}1+\hat\chi P(\tau(s))\dd s}} \\
		&=\dfrac{P'(z)}{\hat\chi(U(z)-P(z))+(1+\hat\chi)\mathcal T(U)(z)}\xrightarrow[{z\to 0^-}]{} \frac{P'(0)}{\hat\chi U(0^-)+1}.
	\end{align*}
	Therefore, by using L'H\^opital's rule, $\mathcal T(U)'(z)$ admits a limit when $z\to 0^-$ and 
	\begin{equation*}
		\lim_{z\to 0^-}\mathcal T(U)'(z)=\frac{P'(0)}{1+\hat\chi }\frac{1+\hat\chi P(0)}{1+\hat\chi U(0^-)}.
	\end{equation*}
\end{proof}
\begin{lem}[Compactness of $\mathcal T$]\label{LEM:4}
	Let Assumption \ref{as:hatchi} hold. The metric space $\mathcal A$ equipped with the distance induced by the $\Vert \cdot\Vert_{\eta}$ norm (defined in \eqref{eq:norm-eta}) is a complete metric space on which the map $\mathcal T:\mathcal A\to\mathcal A$ is  compact.
\end{lem}
\begin{proof}
	Let us first briefly recall that the space $\mathcal A$ is complete. Let $B_\eta$ be the set of all continuous functions defined on $(-\infty, 0)$ with finite $\Vert \cdot\Vert_{\eta}$ norm:
	\begin{equation*}
		B_\eta:=\{u\in C^0\big((-\infty, 0)\big)\,|\, \Vert u\Vert_{\eta}<+\infty\}.
	\end{equation*}
	It is classical that $B_{\eta}$ equipped with the norm $\Vert \cdot\Vert_{\eta}$ is a Banach space. Therefore, in order to prove the completeness of $\mathcal A$,  it suffices to show that $\mathcal A$ is closed in $B_{\eta}$. Let $U_n\in \mathcal A$, $U\in B_{\eta}$ be such that $\lim \Vert U_n-U\Vert_\eta=0$. Then $U_n$ converges to $U$ locally uniformly on $(-\infty, 0)$, and in particular we have 
	\begin{gather*}
		U(z)\in \left[\frac{2}{2+\hat\chi},1 \right] \text{ for all } z\leq 0, \\
		U \text{ is non-increasing}.
	\end{gather*}
	Therefore $u\in \mathcal A$ and the completeness is proved. \medskip

	Let us show that $\mathcal T$ is a compact map of the metric space $\mathcal A$. We have shown in Lemma \ref{LEM:2} that $\mathcal T$ is continuous on $\mathcal A$ and leaves $\mathcal A$ stable.  Let $U_n\in \mathcal A$, then combining Equation \eqref{eq:Vprime} and the local uniform lower bound of $P'(z)-P'(0)$ from Lemma \ref{LEM:1}, the family $\mathcal T(U_n)'|_{[-k, -1/k]}$ is uniformly  Lipschitz continuous  on $[-k, -1/k]$ for each $k\in\mathbb N$. Therefore the Ascoli-Arzelà applies and the set $\{\mathcal T(U_n)|_{[-k, -1/k]}\}_{n\geq 0}$ is relatively compact for the uniform topology on $[-k, -1/k]$ for each $k\in\mathbb N$. Using a diagonal extraction process, there exists a subsequence $\varphi(n)$ and a continuous function $U$ such that $U_{\varphi(n)}\to U$  uniformly on every compact subset of $(-\infty, 0)$. Let us show that $\Vert U_{\varphi(n)}-U\Vert_{\eta}\to 0$ as $n\to+\infty$. Let $\varepsilon>0$ be given, and let $z_0$, $z_1$ be respectively the smallest and largest root of the equation:
	\begin{equation*}
		\eta z+\frac{1}{2}\ln(-z)=\ln\left(\frac{\varepsilon}{2}\right).
	\end{equation*}
	Then, on the one hand,  for any $z\not\in [z_0, z_1]$, we have $\sqrt{-z}e^{\eta z}\leq \frac{\varepsilon}{2}$ and therefore 
	\begin{equation*}
		\sqrt{-z}e^{\eta z}|\mathcal T(U_{\varphi(n)})(z)-\mathcal{T}(U)(z)|\leq \sqrt{-z}e^{\eta z}(|U_{\varphi(n)}(z)|+|U(z)|)\leq \varepsilon.
	\end{equation*}
	On the other hand, since $\mathcal T(U_{\varphi(n)})$ converges locally uniformly to $\mathcal{T}(U)$, there is $n_0\geq 0$ such that 
	\begin{equation*}
		\sup_{z\in[z_0, z_1]}\sqrt{-z}e^{\eta z}|\mathcal T(U_{\varphi(n)})(z)-\mathcal{T}(U)(z)|\leq \varepsilon, \text{ for all } n\geq n_0.
	\end{equation*}
	We conclude that 
	\begin{equation*}
		\Vert \mathcal T(U_{\varphi(n)})-\mathcal{T}(U)\Vert_{\eta}\leq \varepsilon, 
	\end{equation*}
	for all $n\geq n_0$. The convergence is proved. This ends the proof of Lemma \ref{LEM:4}
\end{proof}

We are now in the position to prove Theorem \ref{thm:sharp-TW}. 
\begin{proof}[Proof of Theorem \ref{thm:sharp-TW}]
	We remark that the set of admissible functions $\mathcal A$ is a nonempty, closed, convex, bounded subset of the Banach space $B_\eta$, and $\mathcal T$ is a continuous compact operator on $\mathcal A$ (Lemma \ref{LEM:4}). Therefore, a direct application of the Schauder fixed-point Theorem (see e.g. \cite[Theorem 2.A p. 57]{Zei-86}) shows that $\mathcal T$ admits a fixed point $U$ in $\mathcal A$:
	\begin{equation*}
		\mathcal T(U)=U.
	\end{equation*}
	Applying Lemma \ref{LEM:2} and \ref{LEM:3}, $U$ is strictly decreasing on $(-\infty, 0)$, $U((-\infty, 0))\subset [\frac{2}{2+\hat\chi}, 1]$, $U$ is $C^1$ on $(-\infty, 0]$ and 
	\begin{equation*}
		\lim_{z\to 0^-}U(z)=\frac{1+\hat\chi P(0)}{1+\hat\chi} \text{ and }\lim_{z\to 0^-}U'(z)=\frac{P'(0)}{1+\hat\chi}\frac{1+\hat\chi P(0)}{1+\hat\chi U(0)}.
	\end{equation*}
	Finally
	\begin{equation}\label{eq:Uprime}
		U'(z)=U(z)\frac{1+\hat\chi P(z)-(1+\hat\chi)U(z)}{\chi\big(P'(0)-P'(z)\big)}, \text{ for all }z<0,
	\end{equation}
	therefore
	\begin{equation*}
		\chi P'(0)U'(z)-\chi P'(z)U'(z)-\chi U(z)P''(z)=U(z)(1-U(z)),\text{ for all }  z<0,
	\end{equation*}
	and finally
	\begin{equation*}
		\chi P'(0)U'(z)-\chi (P'(z)U(z))'=U(z)(1-U(z)), \text{ for all }z<0.
	\end{equation*}
	
	We now prove that  $ U(-\infty):=\lim_{z\to \infty}U(z) =1 $. Since $ U $ is monotone decreasing on $ (-\infty,0) $ and is bounded by $ 1 $ from above,  $ U(-\infty) $ exists and, by a direct application of  Lebesgue's dominated convergence theorem, $P$ also converges to a limit near $-\infty$,  $ P(-\infty)=U(-\infty) $. Therefore $U'(z)\to 0$, $P'(z)\to 0 $ and $P''(z)\to 0$ as $z\to-\infty$. We conclude that 
	\begin{equation*}
		\lim_{z\to-\infty} U(z)(1-U(z))=0,
	\end{equation*}
	which implies that $U(-\infty)=1$.
	
	Let us define  $u(t, x):=U(x-ct) $,  with $c:=-\chi P'(0)$. The characteristics associated with $u(t, x)$ are 
	\begin{equation*}
		\frac{\dd}{\dd t}h(t, x)=-\chi(\rho_x\star u)(t, h(t,x))=\chi (\rho\star U)(h(t, x)-ct)=-\chi P'(h(t,x)-ct), 
	\end{equation*}
	and $u(t,x)$ satisfies for all $x$ such that $h(t, x)-ct<0$: 
	\begin{align*}
		\partial_t u(t, h(t, x))&=\partial_t( U(h(t, x)-ct)) = \left(\frac{\dd}{\dd t} (h(t, x)-ct)  \right) U'(h(t,x)-ct) \\
		&=\chi (-P'(h(t, x)-ct)+P'(0))U'(h(t, x)-ct) \\
		&= u(t, h(t, x))(1+\hat\chi (\rho\star u)(t, h(t,x))-(1+\hat\chi)u(t, h(t,x))).
	\end{align*}
	If $h(t,x)-ct> 0$ then $u(t, h(t,x))=U(h(t,x)-ct)=0$ (locally in $t$) and therefore 
	\begin{equation*}
		\partial_t u(t, h(t, x))=0=u(t, h(t, x))(1+\hat\chi (\rho\star u)(t, h(t,x))-(1+\hat\chi )u(t, h(t,x))).
	\end{equation*}
	Since $\{0\}$ is a negligible  set for the Lebesgue measure, we conclude that $u(t, x)$ is a solution integrated along the characteristics to \eqref{eq:main} and thus $U$ is a traveling wave profile with speed $c=-P'(0)>0$ as defined in Definition \ref{def:TW}. Finally
	\begin{equation*}
		c=-\chi P'(0)=\frac{\chi}{2\sigma }\int_{-\infty}^0 e^{y}U(y)\dd y\in \left(\frac{\chi}{\sigma(2+\hat\chi)}, \frac{\chi}{2\sigma }\right)=\left(\frac{\sigma\hat\chi}{2+\hat\chi}, \frac{\sigma \hat\chi}{2 }\right).
	\end{equation*}
	This finishes the proof of Theorem \ref{thm:sharp-TW}
\end{proof}

\subsection{Non-existence of continuous sharp traveling waves}

\begin{rem}
	This result tells us if $ U $ is a sharp traveling wave solution to \eqref{eq:main}, then it must be discontinuous. This situation is very different from the porous medium case. However, it does not exclude the existence of {\it positive} continuous traveling wave solutions which decay to zero near $+\infty$. In fact, as we will show in the numerical simulations in the later section, we can observe numerically large speed traveling wave solutions that are smooth and strictly positive. 
\end{rem}
\begin{proof}[Proof of Proposition \ref{prop:smooth-TW}]
	We divide the proof in 3 steps.

	\begin{stepping}

		\step We show the estimate \eqref{eq:smooth-speed}.

		Assume by contradiction that there exists $x\in \mathbb R $ such that 
		\begin{equation}\label{eq:zeroY}
			-\chi \int_{\mathbb R}\rho_x(x-y)U(y)\dd y = c.
		\end{equation}
		We let $P(x):=(\rho\star U)(x)=\int_\mathbb R \rho(x-y)U(y)\dd y$. Since $U\in C^0(\mathbb R)$, we have that $P\in C^{2}(\mathbb R)$.  Differentiating, we find that  
		\begin{align*}
			P'(x)&=\int_{\mathbb R}\rho_x(x-y)U(y)\dd y = (\rho'\star U)(x),\\
			\sigma^2 P''(x)&= \int_{\mathbb R} \rho(x-y)U(y)\dd y - U(x) = P(x)-U(x).
		\end{align*}
		 Letting $Y(x):=-\chi(\rho_x\star U)(x)-c = -\chi P'(x)-c$, then $Y\in C^1(\mathbb R)$ and we have
		\begin{equation} \label{eq:diffY}
			Y'(x)=-\chi P''(x)=\hat\chi\big( U(x)-(\rho\star U)(x)\big).
		\end{equation}
		Since $\lim_{x\to+\infty} U(x)=0$, we have $\lim_{x\to+\infty}Y(x)=-c<0$. Remark that by our assumption \eqref{eq:zeroY}, $Y$ has at least one zero  and therefore  the largest root of $Y$ is well-defined:
		\begin{equation*}
			x_*:=\inf\{x\,|\,\forall y>x,  Y(y)<0\}.
		\end{equation*}

		We first remark that 
		\begin{equation*}
			\frac{\dd \phantom{t}}{\dd t}\big(h(t, x)-ct\big)=\frac{\dd\phantom{t}}{\dd t}h(t, x)-c = -\chi(\rho_x\star u)(t, h(t, x))-c = Y(h(t, x)-ct),
		\end{equation*}
		where we recall that $u(t, x):=U(x-ct) $ is a solution to \eqref{eq:main}.
		In particular since $Y(x_*)=0$ by the continuity of $Y$, we have $ h(t, x_*)-ct=x_*$.
		Next by  using  \eqref{eq:integrated} we have  
		\begin{align*}
			\frac{\dd\phantom{t}}{\dd t}u(t, h(t, x_*))&=u(t, h(t, x_*))\big(1+\hat\chi (\rho\star u)(t, h(t,x_*))-(1+\hat\chi)u(t, h(t, x_*))\big)\\
			&=U(h(t, x_*)-ct)\big(1+\hat\chi (\rho\star U)(h(t, x_*)-ct)-(1+\hat\chi)U(h(t, x_*)-ct)\big) \\
			&=U(x_*)\big(1+\hat\chi P(x_*)-(1+\hat\chi)U(x_*)\big),
		\end{align*}
		and since $u(t, h(t, x_*))=U(h(t, x_*)-ct)=U(x_*) $ does not depend on $t$, this yields 
		\begin{equation*}
			0=U(x_*)\big(1+\hat\chi P(x_*)-(1+\hat\chi)U(x_*)\big).
		\end{equation*}
		We conclude that either $U(x_*)=0$ or $U(x_*)=\frac{1+\hat\chi P(x_*)}{1+\hat\chi}>0$. In the remaining part of this step  we will show that these two cases lead to contradiction.\medskip

		{\bf Case 1: $U(x_*)=\frac{1+\hat\chi P(x_*)}{1+\hat\chi}>0$.}  By \eqref{eq:diffY} we have: 
		\begin{equation*}
			Y'(x_*)=\hat\chi\big(U(x_*)-P(x_*)\big)=(1-P(x_*))\frac{\hat\chi}{1+\hat\chi},
		\end{equation*}
		however $U(x)\in [0,1]$, $U(x)\not\equiv 1$ and thus $P(x_*)=(\rho\star U)(x_*)<1$ which shows $Y'(x_*)>0$. Yet by definition of $x_*$ we have $Y(x_*)=0$ and $Y(x)<0$ for all $x>x_*$, hence $Y'(x_*)\leq 0$, which is a contradiction.\medskip
		
		{\bf Case 2:  $U(x_*)=0$.} By \eqref{eq:diffY} we have 
		\begin{equation}\label{eq:Y'negative}
			Y'(x_*)=0-\hat\chi P(x_*)=-\hat\chi (\rho\star U)(x_*)<0.
		\end{equation}
		Hence by the continuity of $ Y $, there exists a $ x_0<x_* $ such that 
		\[ Y(x)>0,\quad \forall x\in [x_0,x_*).\]	
		Since for any $ t>0 $, we have
		\begin{equation*}
					\diff{t}(h(t, x_0)-ct)= Y(h(t, x_0)-ct)>0,
		\end{equation*}	
		the function  $t\mapsto h(t,x_0)-ct $ is increasing and converges to $x_*$ as $t\to +\infty$. In particular as $t\to +\infty $ we have $u(t, h(t, x_0))=U(h(t, x_0)-ct)\to U(x_*)=0$. Let $T>0$ be such that $0<u(t, h(t, x_0))\leq \frac{1}{2(1+\hat\chi)}$ for all $t\geq T$. We have
		\begin{align*}
			\frac{\dd\phantom{t}}{\dd t}u(t, h(t, x_0))&=u(t, h(t, x_0))\big(1+\hat\chi (\rho\star u)(t, h(t,x_0))-(1+\hat\chi)u(t, h(t,x_0))\big)\\
			&\geq \frac{1}{2}u(t, h(t, x_0)),
		\end{align*}
		hence $u(t, h(t, x_0))\geq u(T, h(T,x_0))e^{\frac{t-T}{2}}$. In particular letting 
		\begin{equation*}
			t^*:=T-2\ln\big({u(T, h(T,x_0))}\big)>T,
		\end{equation*}
		we have
		\begin{equation*}
			u\left(t^*, h(t^*, x_0)\right)\geq 1>\frac{1}{2(1+\hat\chi)},
		\end{equation*}
		which is a contradiction. Since both Case 1 and Case 2 lead to contradiction, we have shown \eqref{eq:smooth-speed}.

		\step Regularity of $u$. 

		We have shown in Step 1 that for all $x\in \mathbb R $ the strict inequality:
		\begin{equation*}
			Y(x)=-\chi P'(x)-c < 0
		\end{equation*}
		holds.
		Let $x\in \mathbb R$ and $t_0>0$. Then, there exists $y\in\mathbb R$ such that $h(t_0, y)=x$, where $h$ is the characteristic semiflow defined by \eqref{eq:characteristics}.  Since
		\begin{equation*}
			\frac{\dd\phantom{t}}{\dd t}(h(t, y)-ct)=-\chi(\rho_x\star u)(t, h(t, y))-c=Y(h(t, y))\neq 0,
		\end{equation*}
		the mapping $t\mapsto h(t, y)-ct$ has a $C^1$ inverse which we denote $\varphi(z)$, i.e. 
		\begin{equation*}
			\forall z\,|\, \exists t>0, z=h(t, y)-ct, \qquad h(\varphi(z), y)-c\varphi(z)=z. 
		\end{equation*}
		Then we have
		\begin{equation*}
			U(h(t, y)-ct)=u(t, h(t, y))\qquad \Leftrightarrow \qquad U(z)=u(\varphi(z), h(\varphi(z),y)),
		\end{equation*}
		with $z=h(t, y)$ in a neighbourhood of $x$. Since $\varphi $ is $C^1$ and the function $t\mapsto u(t, h(t,y))$ is $C^1$, we conclude that $U$ is $C^1$ in a neighbourhood of $x$. The regularity is proved.

		\step We show that $u$ is positive.

		Combining Step 1 and 2, we know that  $u$ is a classical solution to the equation:
		\begin{align*}
			-cU_x-\chi((\rho\star U)_xU)_x&=U(1-U) \\
			(-c-\chi P') U_x&=U(1+\hat\chi P-(1+\hat\chi)U) \\
			  U_x&=\frac{U}{Y}(1+\hat \chi P-(1+\hat\chi)U),
		\end{align*}
		and since $Y < 0$, the right-hand side is a locally Lipschitz vector field in the variable $U$. In particular, the classical Cauchy-Lipschitz Theorem applies and the only solution with $U(x)=0$ for some $x\in\mathbb R$  is $U\equiv 0$. Since $U$ is non-trivial by assumption, $U$ has to be positive.
	\end{stepping}
\end{proof}

\newpage 
	 \appendix
	 \begin{center}
		 \LARGE\textbf{Appendix}
	 \end{center}
	 
	 \section{An nonlinear function}
	 We study a function used in the proof of Lemma \ref{LEM:1} and Assumption \ref{as:hatchi}.
	 \begin{lem}\label{lem:B1}
		 The function  
		 \begin{equation*}
			 f(x):= \ln\left(\frac{2-x}{x}\right)+\frac{2}{2+x}\left(\frac{x}{2}\ln\left(\frac{x}{2}\right)+1-\frac{x}{2}\right)
		 \end{equation*}
		 defined for $x\in (0,2)$ is strictly decreasing and satisfies 
		 \begin{align*}
			 \lim_{x\to 0^+}f(x)&=+\infty, & \lim_{x\to 2^-}f(x)&=-\infty.
		 \end{align*}
		 In particular $f(x)$ has a unique root in $(0,2)$.

		 Finally, we have $f(1)>0$.
	 \end{lem}
	 \begin{proof}
	 The behavior of $f$ at the boundary is standard. The strict monotony requires the computation of the derivative:
		 \begin{align*}
			 f'(x) &= \left(\frac{-x-(2-x)}{x^2}\right)\times \frac{x}{2-x} + \frac{-2}{(2+x)^2}\left(\frac{x}{2}\ln\left(\frac{x}{2}\right)+1-\frac{x}{2}\right) + \frac{2}{2+x}\ln\left(\frac{x}{2}\right).
		 \end{align*}
		 Recalling that 
		 \begin{equation}\label{B1}
			 \frac{\hat\chi}{2}\ln\left(\frac{\hat\chi}{2}\right)+1-\frac{\hat\chi}{2}> 0,
		 \end{equation}
		 for each $x\in (0,2)$ because $x\mapsto x\ln(x)$ is strictly convex, all three terms in the expression of $f'(x)$ are negative, therefore 
		 \begin{equation*}
			 f'(x)<0
		 \end{equation*}
		 for all $x\in (0,2)$. The fact that $f(1)>0$ can also be deduced from \eqref{B1}. Lemma \ref{lem:B1} is proved.
	 \end{proof}

 \end{document}